\documentclass{article}

\usepackage[hyphens]{url} 
\usepackage[pagebackref,colorlinks,citecolor=darkgreen,linkcolor=darkgreen,unicode]{hyperref}
\usepackage[english]{babel}

\usepackage[utf8]{inputenc}

\usepackage{multicol}

\usepackage{mathpazo}
\usepackage[scaled=0.95]{helvet}
\usepackage{courier}
\usepackage{graphicx}
\usepackage{comment}

\usepackage{fancyhdr} 

\usepackage{ifthen}

\usepackage{amssymb,amsmath,stmaryrd,mathrsfs,wasysym} 
\usepackage{amsthm}

\usepackage{enumitem,mathtools,xspace,xcolor}
\definecolor{darkgreen}{rgb}{0,0.45,0}
\usepackage{aliascnt}
\usepackage[capitalize]{cleveref}
\usepackage{braket} 
\usepackage{tikz-cd}
\usepackage{tikz}
\usetikzlibrary{decorations.pathmorphing}
\usepackage[inference]{semantic}
\usepackage{booktabs}

\usepackage{tocbibind}

\makeatletter
\newcommand{\sectionNotes}{\phantomsection\section*{Notes}\addcontentsline{toc}{section}{Notes}\markright{\textsc{\@chapapp{} \thechapter{} Notes}}}
\newcommand{\sectionExercises}[1]{\phantomsection\section*{Exercises}\addcontentsline{toc}{section}{Exercises}\markright{\textsc{\@chapapp{} \thechapter{} Exercises}}}
\makeatother

\newcommand{\jdeq}{\equiv}      

\newcommand{\defeq}{\vcentcolon\equiv}  

\makeatletter
\def\prd#1{\@ifnextchar\bgroup{\prd@parens{#1}}{\@ifnextchar\sm{\prd@parens{#1}\@eatsm}{\prd@noparens{#1}}}}
\def\prd@parens#1{\@ifnextchar\bgroup%
  {\mathchoice{\@dprd{#1}}{\@tprd{#1}}{\@tprd{#1}}{\@tprd{#1}}\prd@parens}%
  {\@ifnextchar\sm%
    {\mathchoice{\@dprd{#1}}{\@tprd{#1}}{\@tprd{#1}}{\@tprd{#1}}\@eatsm}%
    {\mathchoice{\@dprd{#1}}{\@tprd{#1}}{\@tprd{#1}}{\@tprd{#1}}}}}
\def\@eatsm\sm{\sm@parens}
\def\prd@noparens#1{\mathchoice{\@dprd@noparens{#1}}{\@tprd{#1}}{\@tprd{#1}}{\@tprd{#1}}}
\def\lprd#1{\@ifnextchar\bgroup{\@lprd{#1}\lprd}{\@@lprd{#1}}}
\def\@lprd#1{\mathchoice{{\textstyle\prod}}{\prod}{\prod}{\prod}({\textstyle #1})\;}
\def\@@lprd#1{\mathchoice{{\textstyle\prod}}{\prod}{\prod}{\prod}({\textstyle #1}),\ }
\def\tprd#1{\@tprd{#1}\@ifnextchar\bgroup{\tprd}{}}
\def\@tprd#1{\mathchoice{{\textstyle\prod_{(#1)}}}{\prod_{(#1)}}{\prod_{(#1)}}{\prod_{(#1)}}}
\def\dprd#1{\@dprd{#1}\@ifnextchar\bgroup{\dprd}{}}
\def\@dprd#1{\prod_{(#1)}\,}
\def\@dprd@noparens#1{\prod_{#1}\,}

\def\lam#1{{\lambda}\@lamarg#1:\@endlamarg\@ifnextchar\bgroup{.\,\lam}{.\,}}
\def\@lamarg#1:#2\@endlamarg{\if\relax\detokenize{#2}\relax #1\else\@lamvar{\@lameatcolon#2},#1\@endlamvar\fi}
\def\@lamvar#1,#2\@endlamvar{(#2\,{:}\,#1)}
\def\@lameatcolon#1:{#1}

\def\lamu#1{{\lambda}\@lamuarg#1:\@endlamuarg\@ifnextchar\bgroup{.\,\lamu}{.\,}}
\def\@lamuarg#1:#2\@endlamuarg{#1}

\def\fall#1{\forall (#1)\@ifnextchar\bgroup{.\,\fall}{.\,}}

\def\exis#1{\exists (#1)\@ifnextchar\bgroup{.\,\exis}{.\,}}

\def\sm#1{\@ifnextchar\bgroup{\sm@parens{#1}}{\@ifnextchar\prd{\sm@parens{#1}\@eatprd}{\sm@noparens{#1}}}}
\def\sm@parens#1{\@ifnextchar\bgroup%
  {\mathchoice{\@dsm{#1}}{\@tsm{#1}}{\@tsm{#1}}{\@tsm{#1}}\sm@parens}%
  {\@ifnextchar\prd%
    {\mathchoice{\@dsm{#1}}{\@tsm{#1}}{\@tsm{#1}}{\@tsm{#1}}\@eatprd}%
    {\mathchoice{\@dsm{#1}}{\@tsm{#1}}{\@tsm{#1}}{\@tsm{#1}}}}}
\def\@eatprd\prd{\prd@parens}
\def\sm@noparens#1{\mathchoice{\@dsm@noparens{#1}}{\@tsm{#1}}{\@tsm{#1}}{\@tsm{#1}}}
\def\lsm#1{\@ifnextchar\bgroup{\@lsm{#1}\lsm}{\@@lsm{#1}}}
\def\@lsm#1{\mathchoice{{\textstyle\sum}}{\sum}{\sum}{\sum}({\textstyle #1})\;}
\def\@@lsm#1{\mathchoice{{\textstyle\sum}}{\sum}{\sum}{\sum}({\textstyle #1}),\ }
\def\tsm#1{\@tsm{#1}\@ifnextchar\bgroup{\tsm}{}}
\def\@tsm#1{\mathchoice{{\textstyle\sum_{(#1)}}}{\sum_{(#1)}}{\sum_{(#1)}}{\sum_{(#1)}}}
\def\dsm#1{\@dsm{#1}\@ifnextchar\bgroup{\dsm}{}}
\def\@dsm#1{\sum_{(#1)}\,}
\def\@dsm@noparens#1{\sum_{#1}\,}

\def\wtype#1{\@ifnextchar\bgroup%
  {\mathchoice{\@twtype{#1}}{\@twtype{#1}}{\@twtype{#1}}{\@twtype{#1}}\wtype}%
  {\mathchoice{\@twtype{#1}}{\@twtype{#1}}{\@twtype{#1}}{\@twtype{#1}}}}
\def\lwtype#1{\@ifnextchar\bgroup{\@lwtype{#1}\lwtype}{\@@lwtype{#1}}}
\def\@lwtype#1{\mathchoice{{\textstyle\mathsf{W}}}{\mathsf{W}}{\mathsf{W}}{\mathsf{W}}({\textstyle #1})\;}
\def\@@lwtype#1{\mathchoice{{\textstyle\mathsf{W}}}{\mathsf{W}}{\mathsf{W}}{\mathsf{W}}({\textstyle #1}),\ }
\def\twtype#1{\@twtype{#1}\@ifnextchar\bgroup{\twtype}{}}
\def\@twtype#1{\mathchoice{{\textstyle\mathsf{W}_{(#1)}}}{\mathsf{W}_{(#1)}}{\mathsf{W}_{(#1)}}{\mathsf{W}_{(#1)}}}
\def\dwtype#1{\@dwtype{#1}\@ifnextchar\bgroup{\dwtype}{}}
\def\@dwtype#1{\mathsf{W}_{(#1)}\,}

\def\wtypeh#1{\@ifnextchar\bgroup%
  {\mathchoice{\@lwtypeh{#1}}{\@twtypeh{#1}}{\@twtypeh{#1}}{\@twtypeh{#1}}\wtypeh}%
  {\mathchoice{\@@lwtypeh{#1}}{\@twtypeh{#1}}{\@twtypeh{#1}}{\@twtypeh{#1}}}}
\def\lwtypeh#1{\@ifnextchar\bgroup{\@lwtypeh{#1}\lwtypeh}{\@@lwtypeh{#1}}}
\def\@lwtypeh#1{\mathchoice{{\textstyle\mathsf{W}^h}}{\mathsf{W}^h}{\mathsf{W}^h}{\mathsf{W}^h}({\textstyle #1})\;}
\def\@@lwtypeh#1{\mathchoice{{\textstyle\mathsf{W}^h}}{\mathsf{W}^h}{\mathsf{W}^h}{\mathsf{W}^h}({\textstyle #1}),\ }
\def\twtypeh#1{\@twtypeh{#1}\@ifnextchar\bgroup{\twtypeh}{}}
\def\@twtypeh#1{\mathchoice{{\textstyle\mathsf{W}^h_{(#1)}}}{\mathsf{W}^h_{(#1)}}{\mathsf{W}^h_{(#1)}}{\mathsf{W}^h_{(#1)}}}
\def\dwtypeh#1{\@dwtypeh{#1}\@ifnextchar\bgroup{\dwtypeh}{}}
\def\@dwtypeh#1{\mathsf{W}^h_{(#1)}\,}

\makeatother

\newcommand{\proj}[1]{\ensuremath{\mathsf{pr}_{#1}}\xspace}

\newcommand{\pairr}[1]{{\mathopen{}(#1)\mathclose{}}}

\newcommand{\im}{\ensuremath{\mathsf{im}}} 

\newcommand{\idsym}{{=}}
\newcommand{\id}[3][]{\ensuremath{#2 =_{#1} #3}\xspace}

\newcommand{\idtypevar}[1]{\ensuremath{\mathsf{Id}_{#1}}\xspace}

\newcommand{\refl}[1]{\ensuremath{\mathsf{refl}_{#1}}\xspace}

\newcommand{\ct}{%
  \mathchoice{\mathbin{\raisebox{0.5ex}{$\displaystyle\centerdot$}}}%
             {\mathbin{\raisebox{0.5ex}{$\centerdot$}}}%
             {\mathbin{\raisebox{0.25ex}{$\scriptstyle\,\centerdot\,$}}}%
             {\mathbin{\raisebox{0.1ex}{$\scriptscriptstyle\,\centerdot\,$}}}
}

\newcommand{\opp}[1]{\mathord{{#1}^{-1}}}

\newcommand{\trans}[2]{\ensuremath{{#1}_{*}\mathopen{}\left({#2}\right)\mathclose{}}\xspace}

\newcommand{\mapfunc}[1]{\ensuremath{\mathsf{ap}_{#1}}\xspace} 

\newcommand{\mapdepfunc}[1]{\ensuremath{\mathsf{apd}_{#1}}\xspace} 

\let\ap\map

\let\apd\mapdep

\newcommand{\idfunc}[1][]{\ensuremath{\mathsf{id}_{#1}}\xspace}

\newcommand{\htpy}{\sim}

\newcommand{\eqv}[2]{\ensuremath{#1 \simeq #2}\xspace}

\newcommand{\eqvsym}{\simeq}    

\newcommand{\isequiv}{\ensuremath{\mathsf{isequiv}}}

\newcommand{\hfib}[2]{{\mathsf{fib}}_{#1}(#2)}

\let\type\UU
\newcommand{\typele}[1]{\ensuremath{{#1}\text-\mathsf{Type}}\xspace}

\renewcommand{\set}{\ensuremath{\mathsf{Set}}\xspace}

\newcommand{\prop}{\ensuremath{\mathsf{Prop}}\xspace}

\newcommand{\pointed}[1]{\ensuremath{#1_\bullet}}

\newcommand{\iscontr}{\ensuremath{\mathsf{isContr}}}

\newcommand{\trunc}[2]{\mathopen{}\left\Vert #2\right\Vert_{#1}\mathclose{}}

\newcommand{\tproj}[3][]{\mathopen{}\left|#3\right|_{#2}^{#1}\mathclose{}}

\newcommand{\brck}[1]{\trunc{}{#1}}

\newcommand{\emptyt}{\ensuremath{\mathbf{0}}\xspace}

\newcommand{\unit}{\ensuremath{\mathbf{1}}\xspace}
\newcommand{\ttt}{\ensuremath{\star}\xspace}

\newcommand{\bool}{\ensuremath{\mathbf{2}}\xspace}
\newcommand{\btrue}{{1_{\bool}}}
\newcommand{\bfalse}{{0_{\bool}}}

\newcommand{\inlsym}{{\mathsf{inl}}}
\newcommand{\inrsym}{{\mathsf{inr}}}
\newcommand{\inl}{\ensuremath\inlsym\xspace}
\newcommand{\inr}{\ensuremath\inrsym\xspace}

\newcommand{\glue}{\mathsf{glue}}

\newcommand{\blank}{\mathord{\hspace{1pt}\text{--}\hspace{1pt}}}

\newcommand{\N}{\ensuremath{\mathbb{N}}\xspace}

\newcommand{\istype}[1]{\mathsf{is}\mbox{-}{#1}\mbox{-}\mathsf{type}}

\newcommand{\atMostOne}{\mathsf{atMostOne}}

\makeatletter
\def\defthm#1#2#3{%
  \newaliascnt{#1}{thm}
  \newtheorem{#1}[#1]{#2}
  \aliascntresetthe{#1}
  \crefname{#1}{#2}{#3}}

\newtheorem{thm}{Theorem}[section]
\crefname{thm}{Theorem}{Theorems}
\defthm{cor}{Corollary}{Corollaries}
\defthm{lem}{Lemma}{Lemmas}
\defthm{axiom}{Axiom}{Axioms}
\theoremstyle{definition}
\defthm{defn}{Definition}{Definitions}
\theoremstyle{remark}
\defthm{rmk}{Remark}{Remarks}
\defthm{eg}{Example}{Examples}
\defthm{egs}{Examples}{Examples}
\defthm{notes}{Notes}{Notes}

\crefformat{section}{\S#2#1#3}
\Crefformat{section}{Section~#2#1#3}
\crefrangeformat{section}{\S\S#3#1#4--#5#2#6}
\Crefrangeformat{section}{Sections~#3#1#4--#5#2#6}
\crefmultiformat{section}{\S\S#2#1#3}{ and~#2#1#3}{, #2#1#3}{ and~#2#1#3}
\Crefmultiformat{section}{Sections~#2#1#3}{ and~#2#1#3}{, #2#1#3}{ and~#2#1#3}
\crefrangemultiformat{section}{\S\S#3#1#4--#5#2#6}{ and~#3#1#4--#5#2#6}{, #3#1#4--#5#2#6}{ and~#3#1#4--#5#2#6}
\Crefrangemultiformat{section}{Sections~#3#1#4--#5#2#6}{ and~#3#1#4--#5#2#6}{, #3#1#4--#5#2#6}{ and~#3#1#4--#5#2#6}

\crefformat{appendix}{Appendix~#2#1#3}
\Crefformat{appendix}{Appendix~#2#1#3}
\crefrangeformat{appendix}{Appendices~#3#1#4--#5#2#6}
\Crefrangeformat{appendix}{Appendices~#3#1#4--#5#2#6}
\crefmultiformat{appendix}{Appendices~#2#1#3}{ and~#2#1#3}{, #2#1#3}{ and~#2#1#3}
\Crefmultiformat{appendix}{Appendices~#2#1#3}{ and~#2#1#3}{, #2#1#3}{ and~#2#1#3}
\crefrangemultiformat{appendix}{Appendices~#3#1#4--#5#2#6}{ and~#3#1#4--#5#2#6}{, #3#1#4--#5#2#6}{ and~#3#1#4--#5#2#6}
\Crefrangemultiformat{appendix}{Appendices~#3#1#4--#5#2#6}{ and~#3#1#4--#5#2#6}{, #3#1#4--#5#2#6}{ and~#3#1#4--#5#2#6}

\crefname{part}{Part}{Parts}

\crefname{figure}{Figure}{Figures}

\let\autoref\cref

\let\c@equation\c@thm
\numberwithin{equation}{section}

\setitemize[1]{leftmargin=2em}
\setenumerate[1]{leftmargin=*}

\setitemize[1]{itemsep=-0.2em}
\setenumerate[1]{itemsep=-0.2em}

\def\noteson{%
\gdef\note##1{\mbox{}\marginpar{\color{blue}\textasteriskcentered\ ##1}}}

\noteson

\newcommand{\Coq}{\textsc{Coq}\xspace}

\newcommand{\choice}[1]{\ensuremath{\mathsf{AC}_{#1}}\xspace}

\newcounter{symindex}

\makeatletter

\renewcommand{\id}[3][]{
  \@ifnextchar\bgroup
    {#2 \mathbin{\idsym_{#1}} \id[#1]{#3}}
    {#2 \mathbin{\idsym_{#1}} #3}
  }

\renewcommand{\eqv}[2]{
  \@ifnextchar\bgroup
    {#1 \eqvsym \eqv{#2}}
    {#1 \eqvsym #2}
  }

\newcommand{\ctsym}{%
  \mathchoice{\mathbin{\raisebox{0.5ex}{$\displaystyle\centerdot$}}}%
             {\mathbin{\raisebox{0.5ex}{$\centerdot$}}}%
             {\mathbin{\raisebox{0.25ex}{$\scriptstyle\,\centerdot\,$}}}%
             {\mathbin{\raisebox{0.1ex}{$\scriptscriptstyle\,\centerdot\,$}}}
  }

\renewcommand{\ct}[3][]{
  \@ifnextchar\bgroup
    {#2 \mathbin{\ctsym_{#1}} \ct[#1]{#3}}
    {#2 \mathbin{\ctsym_{#1}} #3}
  }

\makeatother

\makeatletter
\renewcommand{\@dprd}{\@tprd}
\renewcommand{\@dsm}{\@tsm}
\renewcommand{\@dprd@noparens}{\@tprd}
\renewcommand{\@dsm@noparens}{\@tsm}

\renewcommand{\@tprd}[1]{\mathchoice{{\textstyle\prod_{(#1)}\,}}{\prod_{(#1)}\,}{\prod_{(#1)}\,}{\prod_{(#1)}\,}}
\renewcommand{\@tsm}[1]{\mathchoice{{\textstyle\sum_{(#1)}\,}}{\sum_{(#1)}\,}{\sum_{(#1)}\,}{\sum_{(#1)}\,}}

\newcommand{\implicitargumentson}{\boolean{true}}

\newcommand{\@ifnextchar@starorbrace}[2]
  {\@ifnextchar*{#1}{\@ifnextchar\bgroup{#1}{#2}}}
  
\renewcommand{\prd}{\@ifnextchar*{\@iprd}{\@prd}}

\newcommand{\@prd}[1]
  {\@ifnextchar@starorbrace
    {\prd@parens{#1}}
    {\@ifnextchar\sm{\prd@parens{#1}\@eatsm}{\prd@noparens{#1}}}}
\newcommand{\@prd@parens}{\@ifnextchar*{\@iprd}{\prd@parens}}
\renewcommand{\prd@parens}[1]
  {\@ifnextchar@starorbrace
    {\@theprd{#1}\@prd@parens}
    {\@ifnextchar\sm{\@theprd{#1}\@eatsm}{\@theprd{#1}}}}
\newcommand{\@theprd}[1]
  {\mathchoice{\@dprd{#1}}{\@tprd{#1}}{\@tprd{#1}}{\@tprd{#1}}}
\renewcommand{\dprd}[1]{\@dprd{#1}\@ifnextchar@starorbrace{\dprd}{}}
\renewcommand{\tprd}[1]{\@tprd{#1}\@ifnextchar@starorbrace{\tprd}{}}

\newcommand{\@theiprd}[1]{\mathchoice{\@diprd{#1}}{\@tiprd{#1}}{\@tiprd{#1}}{\@tiprd{#1}}}
\newcommand{\@iprd}[2]{\@ifnextchar@starorbrace%
  {\@theiprd{#2}\@prd@parens}%
  {\@ifnextchar\sm%
    {\@theiprd{#2}\@eatsm}%
    {\@theiprd{#2}}}}
\def\@tiprd#1{
  \ifthenelse{\implicitargumentson}
    {\@@tiprd{#1}\@ifnextchar\bgroup{\@tiprd}{}}
    {\@tprd{#1}}}
\def\@@tiprd#1{\mathchoice{{\textstyle\prod_{\{#1\}}\,}}{\prod_{\{#1\}}\,}{\prod_{\{#1\}}\,}{\prod_{\{#1\}}\,}}
\def\@diprd{
  \ifthenelse{\implicitargumentson}
    {\@tiprd}
    {\@tprd}}

\def\@eatprd\prd{\@prd@parens}

\makeatother

\makeatletter
\def\tfall#1{\forall_{(#1)}\@ifnextchar\bgroup{\,\tfall}{\,}}
\renewcommand{\fall}{\tfall}

\def\texis#1{\exists_{(#1)}\@ifnextchar\bgroup{\,\texis}{\,}}
\renewcommand{\exis}{\texis}

\def\uexis#1{\exists!_{(#1)}\@ifnextchar\bgroup{\,\uexis}{\,}}
\makeatother

\newcommand{\typefont}{\mathsf} 

\renewcommand{\isequiv}{\typefont{isEquiv}}

\renewcommand{\pairr}[1]{{\mathopen{}\langle #1\rangle\mathclose{}}}
\renewcommand{\type}{\typefont{Type}}

\newcommand{\mappingcone}[1]{\mathcal{C}_{#1}}

\newcommand{\tfcolim}{\typefont{colim}}

\newcommand{\tfinj}{\typefont{inj}}
\newcommand{\tfsurj}{\typefont{surj}}

\newcommand{\sbrck}[1]{\Vert #1\Vert}
\newcommand{\strunc}[2]{\Vert #2\Vert_{#1}}

\newcommand{\smal}{\mathcal{S}}

\newcommand{\eqrel}{\mathit{EqRel}}
\newcommand{\piw}{\ensuremath{\Pi\typefont{W}}} 
\renewcommand{\sslash}{/\!\!/}

\newcommand{\isepif}{\mathsf{epi}}
\newcommand{\isepi}[1]{\isepif(#1)}

\newcommand{\Nnegtwo}{\N_{-2}}
\newcommand{\stproj}[2]{|#2|_{#1}}
\newcommand{\sbproj}[1]{\stproj{}{#1}}
\newcommand{\coeq}[3]{#1/_{\!#2,#3}}
\newcommand{\eq}[2]{\typefont{eq}_{#1,#2}}
\newcommand{\catset}{\mathit{Set}}

\makeatletter
\newcommand{\jctx}{\@ifnextchar*{\@jctxAlignTrue}{\@jctxAlignFalse}}
\newcommand{\@jctxAlignTrue}[2]{& \vdash #2~ctx}
\newcommand{\@jctxAlignFalse}[1]{\vdash #1~ctx}

\newcommand{\jtype}{\@ifnextchar*{\@jtypeAlignTrue}{\@jtypeAlignFalse}}
\newcommand{\@jtypeAlignFalse}[2]{#1\vdash #2~type}
\newcommand{\@jtypeAlignTrue}[3]{#2 & \vdash #3~type}

\newcommand{\jtermc}{\@ifnextchar*{\@jtermcAlignTrue}{\@jtermcAlignFalse}}
\newcommand{\@jtermcAlignTrue}[3]{\@jtermAlignTrue{#1}{}{#2}{#3}}
\newcommand{\@jtermcAlignFalse}[2]{\@jtermAlignFalse{}{#1}{#2}}

\newcommand{\jterm}{\@ifnextchar*{\@jtermAlignTrue}{\@jtermAlignFalse}}
\newcommand{\@jtermAlignTrue}[4]{#2 & \vdash #4:#3}
\newcommand{\@jtermAlignFalse}[3]{#1 \vdash #3:#2}

\newcommand{\jctxeq}{\@ifnextchar*{\@jctxeqAlignTrue}{\@jctxeqAlignFalse}}
\newcommand{\@jctxeqAlignTrue}[3]{& \vdash #2\jdeq #3~ctx}
\newcommand{\@jctxeqAlignFalse}[2]{\vdash #1\jdeq #2~ctx}

\newcommand{\jtypeeq}{\@ifnextchar*{\@jtypeeqAlignTrue}{\@jtypeeqAlignFalse}}
\newcommand{\@jtypeeqAlignTrue}[4]{#2 & \vdash #3\jdeq #4~type}
\newcommand{\@jtypeeqAlignFalse}[3]{#1\vdash #2\jdeq #3~type}

\newcommand{\jtermceq}{\@ifnextchar*{\@jtermceqAlignTrue}{\@jtermceqAlignFalse}}
\newcommand{\@jtermceqAlignTrue}[4]{\@jtermeqAlignTrue{#1}{}{#2}{#3}{#4}}
\newcommand{\@jtermceqAlignFalse}[3]{\@jtermeqAlignFalse{}{#1}{#2}{#3}}

\newcommand{\jtermeq}{\@ifnextchar*{\@jtermeqAlignTrue}{\@jtermeqAlignFalse}}
\newcommand{\@jtermeqAlignTrue}[5]{#2 & \vdash #4\jdeq #5:#3}
\newcommand{\@jtermeqAlignFalse}[4]{#1\vdash #3\jdeq #4:#2}
\makeatother

\let\inferenceold\inference
\renewcommand{\inference}[2]{\bgroup\renewcommand*{\arraystretch}{1}\inferenceold{#1}{#2}\egroup}

\makeatletter
\newcommand{\ctxext}[2]{\@ctxext@ctx #1.\@ctxext@type #2}
\newcommand{\@ctxext}{\@ifnextchar\bgroup{\@@ctxext}{}}
\newcommand{\@ctxext@ctx}{\@ifnextchar\ctxext{\@ctxext@nested}{\@ifnextchar\ctxwk{\@ctxwk@nested}{\@ctxext}}}
\newcommand{\@ctxext@type}{\@ifnextchar\ctxext{\@ctxext@nested}{\@ifnextchar\subst{\@subst@nested}{\@ctxext}}}
\newcommand{\@@ctxext}[1]{\@ifnextchar\bgroup{\@ctxext@parens{#1}}{#1}}
\newcommand{\@ctxext@parens}[2]{(\ctxext{#1}{#2})}
\newcommand{\@ctxext@nested}[3]{\@ctxext@parens{#2}{#3}}
\makeatother

\makeatletter
\newcommand{\subst}[2]{\@subst@type #2[\@subst@term #1]}
\newcommand{\@subst}{\@ifnextchar\bgroup{\@@subst}{}}
\newcommand{\@@subst}[1]{\@ifnextchar\bgroup{\subst{#1}}{#1}}
\newcommand{\@subst@term}{\@subst}
\newcommand{\@subst@type}{\@ifnextchar\ctxext{\@ctxext@nested}{\@ifnextchar\ctxwk{\@ctxwk@nested}{\@subst}}}
\newcommand{\@subst@nested}[3]{\@subst@parens{#2}{#3}}
\newcommand{\@subst@parens}[2]{(\subst{#1}{#2})}
\makeatother

\makeatletter
\newcommand{\ctxwk}[2]{\langle\@ctxwk@act #1\rangle\@ctxwk@pass #2}
\newcommand{\@ctxwk}{\@ifnextchar\bgroup{\@@ctxwk}{}}
\newcommand{\@@ctxwk}[1]{\@ifnextchar\bgroup{\ctxwk{#1}}{#1}}
\newcommand{\@ctxwk@act}{\@ctxwk}
\newcommand{\@ctxwk@pass}{\@ifnextchar\ctxext{\@ctxext@nested}{\@ifnextchar\subst{\@subst@nested}{\@ctxwk}}}
\newcommand{\@ctxwk@parens}[2]{(\ctxwk{#1}{#2})}
\newcommand{\@ctxwk@nested}[3]{\@ctxwk@parens{#2}{#3}}
\makeatother

\makeatletter
\newcommand{\jhom}[3]{\jterm{\jhom@dom@ctx #1}{\jhom@dom@wk{#1}{#2}}{#3}}

\newcommand{\jhom@dom@ctx}{\@ifnextchar\bgroup{\@jhom@dom@ctx}{}}
\newcommand{\@jhom@dom@ctx}[1]{\@ifnextchar\bgroup{\@@jhom@dom@ctx{#1}}{#1}}
\newcommand{\@@jhom@dom@ctx}[2]{\ctxext{\jhom@dom@ctx #1}{\jhom@dom@wk{#1}{#2}}}

\newcommand{\jhom@dom@wk}[2]{
  \jhom@dom@choice{\jhom@dom@wk@root #1\protect\null}
                  {\jhom@dom@wk@branchone #1\protect\null}
                  {\jhom@dom@wk@branchtwo #1\protect\null}
                  {#2}}
\newcommand{\jhom@dom@choice}[4]{
  \ifx {#1\protect\null}{\protect\null}
  \@jhom@dom@wk@branches{#2}{#3}{#4}
  \else
  \@jhom@dom@wk@root{#1}{#4}
  \fi
}
\newcommand{\@jhom@dom@wk@branches}[3]{\expandafter{\expandafter{\jhom@dom@wk}{#1}}{\jhom@dom@wk{#2}{#3}}}
\newcommand{\@jhom@dom@wk@root}[2]{\ctxwk{#1}{#2}}

\newcommand{\jhom@dom@wk@root}{\@ifnextchar\bgroup{\jhom@dom@wk@eat}{}}

\newcommand{\jhom@dom@wk@branchone}{\@ifnextchar\bgroup{\@jhom@dom@wk@branchone}{\jhom@dom@wk@eat}}
\newcommand{\@jhom@dom@wk@branchone}[1]{#1\jhom@dom@wk@eat}

\newcommand{\jhom@dom@wk@branchtwo}{\@ifnextchar\bgroup{\@jhom@dom@wk@branchtwo}{\jhom@dom@wk@eat}}
\newcommand{\@jhom@dom@wk@branchtwo}[1]{\@ifnextchar\bgroup{\@@jhom@dom@wk@branchtwo}{\jhom@dom@wk@eat}}
\newcommand{\@@jhom@dom@wk@branchtwo}[1]{#1\jhom@dom@wk@eat}

\newcommand{\jhom@dom@wk@eat}[1]{\@ifnextchar{\@let@token\null}{}{\jhom@dom@wk@eat}}
\makeatother

\makeatletter
\newcommand{\pt}[1][]{*_{
  \@ifnextchar\undergraph{\@undergraph@nested}
    {\@ifnextchar\underovergraph{\@underovergraph@nested}{}}#1}}
\makeatother

\makeatletter
\newcommand{\pts}[1]{{\@graphop@nested{#1}}_{0}}
\newcommand{\edg}[1]{{\@graphop@nested{#1}}_{1}}
\newcommand{\@graphop@nested}[1]
  {\@ifnextchar\ctxext{\@ctxext@nested}
      {\@ifnextchar\undergraph{\@undergraph@nested}
         {\@ifnextchar\underovergraph{\@underovergraph@nested}{}}}
    #1}
\makeatother

\makeatletter
\newcommand{\@undergraphtest}[2]{\@ifnextchar({#1}{#2}}
\newcommand{\undergraph}[2]{\@undergraphtest{\@undergraph@parens{#1}{#2}}{\@undergraph{#1}{#2}}}
\newcommand{\@undergraph}[2]{{#2/#1}}
\newcommand{\@undergraph@nested}[3]{\@undergraph@parens{#2}{#3}}
\newcommand{\@undergraph@parens}[2]{(\@undergraph{#1}{#2})}
\makeatother

\makeatletter
\newcommand{\underovergraph}[2]{\@underovergraphtest{\@underovergraph@parens{#1}{#2}}{\@underovergraph{#1}{#2}}}
\newcommand{\@underovergraph}[2]{{#2}\,{\parallel}\,{#1}}
\newcommand{\@underovergraphtest}{\@undergraphtest}
\newcommand{\@underovergraph@parens}[2]{(\@underovergraph{#1}{#2})}
\newcommand{\@underovergraph@nested}[3]{\@underovergraph@parens{#2}{#3}}
\makeatother

\tikzset{patharrow/.style={double,double equal sign distance,-,font=\scriptsize}}
\tikzset{description/.style={fill=white,inner sep=2pt}}

\tikzset{commutative diagrams/column sep/Huge/.initial=18ex}

\defthm{conj}{Conjecture}{Conjectures}

\newlength\minalignvsep

\makeatletter
\def\align@preamble{%
   &\hfil
    \setboxz@h{\@lign$\m@th\displaystyle{##}$}%
    \ifnum\row@>\@ne
    \ifdim\ht\z@>\ht\strutbox@
    \dimen@\ht\z@
    \advance\dimen@\minalignvsep
    \ht\strutbox\dimen@
    \fi\fi
    \strut@
    \ifmeasuring@\savefieldlength@\fi
    \set@field
    \tabskip\z@skip
   &\setboxz@h{\@lign$\m@th\displaystyle{{}##}$}%
    \ifnum\row@>\@ne
    \ifdim\ht\z@>\ht\strutbox@
    \dimen@\ht\z@
    \advance\dimen@\minalignvsep
    \ht\strutbox@\dimen@
    \fi\fi
    \strut@
    \ifmeasuring@\savefieldlength@\fi
    \set@field
    \hfil
    \tabskip\alignsep@
}
\makeatother

\allowdisplaybreaks

\setdescription[1]{itemsep=-0.2em}

\begin{document}
\title{Sets in homotopy type theory}
\author{Egbert Rijke \and Bas Spitters
\thanks{The research leading to these results has received funding from the 
  European Union's $7$th Framework Programme under grant agreement nr.~243847 
  (ForMath).}}
\date{\today}
\maketitle

\begin{abstract}
Homotopy Type Theory may be seen as an internal language 
for the $\infty$-category of weak $\infty$-groupoids which in particular
models the univalence axiom.
Voevodsky proposes this language for weak $\infty$-groupoids as a new foundation for mathematics 
called the Univalent Foundations of Mathematics.
It includes the sets as weak $\infty$-groupoids
with contractible connected components, and thereby it includes (much of) the
traditional set theoretical foundations as a special case. 
We thus wonder whether those `discrete' groupoids
do in fact form a (predicative) topos. More generally, homotopy type theory is 
conjectured to be the internal language of `elementary'
$\infty$-toposes. We prove that sets in homotopy
type theory form a $\piw$-pretopos. This is similar to the fact that the
$0$-truncation of an $\infty$-topos is a topos. We show that both a subobject classifier and a
$0$-object classifier are available for the type theoretical universe of sets.
However, both of these are large and moreover, the $0$-object classifier for
sets is a function between $1$-types (i.e.~groupoids) rather than between sets.
Assuming an impredicative propositional resizing rule we may render the 
subobject classifier small and then we actually obtain a topos of sets.
\end{abstract}

\tableofcontents

\section*{Introduction
  \phantomsection\addcontentsline{toc}{section}{Introduction}}

\emph{A preliminary version of this paper was ready when the standard reference book on homotopy type theory~\cite{TheBook} was produced.
In fact, many of the results of this paper can now also be found in chapter 10 of that book.
Conversely, the collaborative writing of that chapter helped us to clarify the presentation of the present article.
The paper is also meant to give a readable account of some computer proofs which have meanwhile found their way to 
\url{https://github.com/HoTT/HoTT/}.}

Homotopy type theory~\cite{awodey2012type} extends the Curry-Howard correspondence between 
simply typed $\lambda$-calculus, cartesian closed categories and minimal logic, 
via extensional dependent type theory, locally cartesian closed categories and 
predicate logic~\cite{lambek1988introduction,jacobs1999categorical} to Martin-L\"of type theory with identity types and 
certain homotopical models. The univalent foundations program~\cite{kapulkin2012univalence,pelayo2012homotopy,TheBook}
extends homotopy type theory with the so-called univalence axiom, thus providing a language for $\infty$-groupoids.
Voevodsky's insight was that this can serve as a new foundation for mathematics.
The $\infty$-groupoids form the prototypical higher topos~\cite{rezk2010toposes,lurie2009higher}.
In fact, homotopy type theory and the univalence axiom can be interpreted in any such higher topos~\cite{shulman2012univalence}.
It is conjectured~\cite{awodey2012type,shulman2012univalence} that homotopy type theory with univalence and so-called higher inductive 
types can provide an `elementary' definition of a higher topos. 
In this way it extends the previous program to use toposes~\cite{MacLaneMoerdijk,johnstone:elephant} as a foundation of 
mathematics.

In the present article we connect the theory of sets ($0$-types) in the univalent 
foundations with the theory of predicative toposes 
\cite{MoerdijkPalmgren2002,vandenBerg2012}.
The prototypical example of a predicative topos is the category of setoids in 
Martin-L\"of type theory. A setoid~\cite{Bishop1967} is a 
pair of a type with an equivalence relation on it. In this respect, homotopy type theory may be seen as a generalization of the 
rich type theory which may be modeled in setoids~\cite{hofmann1995extensional,Altenkirch1999}.
A setoid is also a groupoid in which all 
hom-sets have at most one inhabitant. Thus groupoids generalize setoids. 
Hofmann and Streicher showed that groupoids form a model for intensional type
theory~\cite{hofmann1998groupoid}. 
In their article Hofmann and Streicher propose to investigate whether 
higher groupoids can also model Martin-L\"of type theory; 
they moreover suggest a form of the univalence axiom for categories:
that isomorphic objects be equal.
The Streicher/Voevodsky's Kan simplicial set model of type 
theory~\cite{kapulkin2012univalence,Streicher:sstt} is such a higher dimensional version 
of Hofmann and Streicher's groupoid model. 
Moreover, Voevodsky and Streicher recognized that the univalence property holds for Kan 
simplicial sets. Voevodsky proposed to investigate Martin-L\"of type theory with the univalence axiom 
as a new foundation for mathematics. Later it was found that the addition of higher inductive types was necessary to model general homotopy 
colimits.

Grothendieck conjectured that Kan simplicial sets and weak $\infty$-groupoids 
are equivalent, however, precisely defining this equivalence is the topic of 
active research around the `Grothendieck homotopy hypothesis'.
As emphasized by Coquand~\cite{Cubical}, both the $0$-truncated weak $\infty$-groupoids and the $0$-truncated Kan 
simplicial sets are similar to setoids, and constructively so. 
This is made precise for instance by the fact that the $0$-truncation of a 
model topos is a Grothendieck topos \cite[Prop 9.2]{rezk2010toposes} and every 
Grothendieck topos arises in this way \cite[Prop 9.4]{rezk2010toposes}. 
We prove an internal version of the former result. 
An internal version of the latter result may follow by carrying out the constructive model construction in a (predicative) 
topos~\cite{Cubical}.

Predicative topos theory follows the methodology of algebraic set 
theory~\cite{JoyalMoerdijk1995}, a categorical treatment of set theory, 
which in particular captures the notion of smallness by considering pullbacks 
of a universally small map. 
It extends the ideas from the elementary theory of the category of sets 
(ETCS)~\cite{lawvere2003sets,palmgren2012constructivist} 
by including a universe. 
It thus seems to be an ideal framework to investigate $0$-types. 
However, there is a catch: under the univalence axiom, the universe of sets 
is itself not a set but a $1$-groupoid. 
Thus the existing framework of algebraic set theory will not be entirely 
sufficient, and needs to be revised for univalent purposes.
The proper treatment of universes is a main reason for preferring the $\infty$-groupoid model over the simpler setoid model.
Obviously, the possibility to do synthetic homotopy theory is another; see \cite{TheBook}.

The paper is organized as follows. 
In \autoref{sec:prelim} we sketch a quick overview of the relevant notions of 
homotopy type theory on which we rely and we provide the reader with background
and precise references to the corresponding results and proofs in 
\cite{TheBook}. 
We give slight generalizations of definitions and results where that is just
as easy as merely giving the results specifically about sets. In doing so, we
provide the reader with a sense to what parts of the theory are specific to sets.

The main body of our article is contained in \autoref{sec:quotient}. 
We begin by proving the principle of unique choice 
if we admit the $(-1)$-truncation operation, and we show that epimorphisms are 
surjective. We then prove that $\catset$ is a regular category. 
When we add quotients, $\catset$ becomes exact and even a $\piw$-pretopos,
which is one of the main results of our paper. 
By adding the univalence axiom we show that the groupoid $\set$ is 
an $0$-object classifier. In fact, we do this by showing that the universe 
$\type$ is an object classifier, followed by showing that $\typele{n}$ is an 
$n$-object classifier for every $n:\N$. This is our other main result. 

In \autoref{sec:multiplechoice} we discuss the representation axiom, 
the collection axiom and the axiom of multiple choice. 
These axioms from algebraic set theory are used in stronger systems for 
predicative topos theory.

We expect the reader to be familiar with type theory, category theory, algebraic 
set theory and basic homotopy theory. All background information can be found in~\cite{TheBook}.

\subsection*{Notations, conventions and assumptions}
In this article we use the notation and conventions developed in~\cite{TheBook} and standard categorical definitions as in
\cite{johnstone:elephant}. 
We use Martin-L\"of intensional type theory with a hierarchy of universes \'a la Russell, universe polymorphism and
typical ambiguity.
This means that a definition like $\set$ really gives a definition at each universe level.
There are only a few places where we need to be explicit about universe levels.
The proof that epis are surjective, \autoref{prop:images_are_coequalizers} is one such place where we need both small sets and a large power set.
As is common in homotopy type theory, we will freely use the axiom of function 
extensionality.
Using this axiom, all limits can be constructed in type 
theory~\cite{Avigad:limits,RijkeSpitters:colims}.
We use higher inductive types to implement 
pushouts~\cite{LumsdaineShulman} and truncations. Although the full details of the computational interpretation of such higher inductive 
types are being worked out, these concrete instances indeed do have a computational interpretation~\cite{Cubical,Barras}.

In our result that $\catset$ is a $\piw$-pretopos, we will also use the univalence 
axiom for mere propositions: for every two $(-1)$-types $P$ and $Q$, the function 
$(\id[\type]{P}{Q})\to(\eqv{P}{Q})$ is an equivalence. Univalence for mere 
propositions is nothing new to a topos theorist: it is equivalent to having a
subobject classifier (although it would be large by default).
It is also used in \autoref{prop:1surj_to_surj_to_pem}.
To show that $\pointed\type\to\type$ is an object classifier 
(\autoref{sec:object_classification}) we will use the (full) univalence axiom.

In the previous paragraph we have used informal categorical terminology such as 
`pushout' to refer to what would be interpreted as homotopy pushouts in the 
$\infty$-category of weak $\infty$-groupoids. 
We will usually omit the word `homotopy' when we talk about those categorical 
concepts and we will continue to use this naive $\infty$-category style without 
any rigorous claims about their interpretation in the model. 
We can be more precise when speaking about $1$-categories: they will be the 
pre-categories that are developed in~\cite{1categories}. 
Assuming the univalence axiom, $\catset$ is in fact a (Rezk complete) $1$-category.

\section{Preliminaries}\label{sec:prelim}

\subsection{The very basics of the univalent foundations}
We denote \emph{identity types} $\mathsf{Id}_A(x,y)$ by $\id[A]{x}{y}$ or simply
by $\id{x}{y}$.
The \emph{concatenation} of $p:\id[A]{x}{y}$ with $q:\id[A]{y}{z}$ will be denoted by
$\ct{p}{q}:\id[A]{x}{z}$. If $P:A\to\type$ is a family of types over $A$ and
$p:\id[A]{x}{y}$ is a path in $A$ and $u:P(x)$, we denote the 
\emph{transportation of $u$ along $p$} by $\trans{p}{u}:P(y)$.
The type $\sm{x:A}\prd{y:A}(\id[A]{y}{x})$ witnessing that a type $A$
is \emph{contractible} is written as $\iscontr(A)$ and the type
$\prd{x:A}\id{f(x)}{g(x)}$ of homotopies from $f:\prd{x:A}P(x)$ to 
$g:\prd{x:A}P(x)$ is written as $f\htpy g$.

A function
$f:A\to B$ is said to be an \emph{equivalence} if there is an element of type
\begin{equation*}
\isequiv(f)\defeq \left(\sm{g:B\to A}g\circ f\htpy\idfunc[A]\right)
\times\left(\sm{h:B\to A}f\circ h\htpy\idfunc[B]\right).
\end{equation*}
The \emph{homotopy fiber} $\sm{a:A}\id[B]{f(a)}{b}$ of $f$ at $b$ is denoted by $\hfib{f}{b}$. 
A function is an equivalence if and only if all its homotopy fibers are contractible. 
We write $\eqv{A}{B}$ for the type $\sm{f:A\to B}\isequiv(f)$ and we usually make
no notational distinction between an equivalence $e:\eqv{A}{B}$ and the underlying
function.

By path induction, the elimination principle for the identity type, there is a canonical function assigning an equivalence $\eqv{A}{B}$ to 
every path
$p:\id[\type]{A}{B}$. The \emph{univalence axiom} asserts that this function is an equivalence
between the types $\id[\type]{A}{B}$ and $\eqv{A}{B}$. The principle
of function extensionality is a consequence of the univalence axiom. 

In the rest of this preliminary chapter we sketch a quick overview of the
theory of homotopy $n$-types and the (co)completeness of those. 

\subsection{Introducing the type of sets}

We define the notion of being an $n$-type by recursion on $\Nnegtwo$,
where $\Nnegtwo$ is a version of the natural numbers which starts at $-2$.
We will be mainly concerned with the type \typele{0} of all $0$-types, which are the sets,
but it is not possible to ignore the other values entirely. A concrete reason
for this is that \typele{0} is itself a $1$-type; see \autoref{lem:ntype_is_sntype}.

\begin{defn}
A type is said to be a \emph{$(-2)$-truncated type} if it is contractible. Thus, we
define 
\begin{equation*}
\istype{(-2)}(A)\defeq \iscontr(A).
\end{equation*}
For $n:\Nnegtwo$, we say that $A$ is \emph{$(n+1)$-truncated} if there is an element
of type
\begin{equation*}
\istype{(n+1)}(A)\defeq \prd{x,y:A}\istype{n}(\id[A]{x}{y}).
\end{equation*}
We also say that $A$ is an \emph{$n$-type} if $A$ is $n$-truncated. We define
\begin{equation*}
\typele{n}\defeq \sm{A:\type}\istype{n}(A).
\end{equation*}
\end{defn}

Relying on the interpretation of identity types as path spaces, 
a useful way of looking at the $n$-truncated types is that $n$-types have no interesting
homotopical structure above truncation level $n$.

In the hierarchy of truncatedness we just introduced, we find the sets in \typele{0}. 
Those are the types $A$ with the property that for any two points $x,y:A$,
the identity type $\id[A]{x}{y}$ is a \emph{mere proposition}: contractible once inhabited. From a topological
point of view, the $0$-truncated spaces are those with contractible connected components. Such a space is homotopically
equivalent to a discrete space.

\begin{defn}
We introduce the following notations:
\begin{equation*}
\prop\defeq \typele{(-1)}\qquad\set \defeq \typele{0}.
\end{equation*}
When $A:\prop$, we also say that $A$ is a \emph{mere proposition}. Although
strictly speaking $\prop$ and $\set$ are dependent pair types, we shall make no
notational distinguishment between terms 
$P:\prop$ and $A:\set$ and their underlying types.
\end{defn}

\begin{rmk}
Strictly speaking, the category of sets is a different entity than the type of
all sets. Indeed, the type $\set$ is the class of objects of the category of sets
and the hom-sets are just the function types. Composition, identity morphisms
and the various necessary identifications are all inherited from general type
theory. The category of sets is denoted by $\catset$.
\end{rmk}

The following theorem provides a useful way to show that types are sets. In \cite{TheBook}, it is used to prove Hedberg's theorem and to prove that the Cauchy-reals are a set.

\begin{thm}[See 7.2.2 in \cite{TheBook}]\label{thm:h-set-refrel-in-paths-sets}
  Suppose $R$ is a reflexive mere relation on a type $X$ implying identity.
  Then $X$ is a set, and $R(x,y)$ is equivalent to $\id{x}{y}$ for all $x,y:X$.
\end{thm}

\begin{proof}
  Let $\rho : \prd{x:X} R(x,x)$ be a proof of reflexivity of $R$,
  and consider a witness $f : \prd{x,y:X} R(x,y) \to (\id{x}{y})$ 
  of the assumption that $R$ implies identity.
  Note first that the two statements in the theorem are equivalent.
  For on one hand, if $X$ is a set, then $\id{x}{y}$ is a mere proposition, 
  and since there is a bi-implication to the mere proposition
  $R(x,y)$ by hypothesis, it must also be equivalent to it.
  On the other hand, if $\id{x}{y}$ is equivalent to $R(x,y)$, 
  then like the latter it is a mere proposition for all $x,y:X$, 
  and hence $X$ is a set.
  
  We show that each $f(x,y) : R(x,y) \to \id{x}{y}$ is an equivalence.
  By Theorem 4.7.7 in \cite{TheBook} it suffices to show that $f$ induces an 
  equivalence of total spaces:
  \begin{equation*}
    \eqv{\big(\sm{y:X}R(x,y)\big)}{\big(\sm{y:X}\id{x}{y}\big)}.
  \end{equation*}
  The type on the right is contractible, so it
  suffices to show that the type on the left is contractible too. As the center of
  contraction we take the pair $\pairr{x,\rho(x)}$.  It remains to show, for
  every ${y:X}$ and every $H:R(x,y)$ that
  \begin{equation*}
    \id{\pairr{x,\rho(x)}}{\pairr{y,H}},
  \end{equation*}
  which is by Theorem 2.7.2 of \cite{TheBook} equivalent to
  \begin{equation*}
  \sm{p:\id{x}{y}}\id{\trans{p}{\rho(x})}{H}.
  \end{equation*}
  But since $R(x,y)$ is a mere proposition, it suffices to show that
  $\id{x}{y}$, which we get from $f(H)$.
\end{proof}

\subsection{Closure properties of the $n$-types}
We list some of the basic properties of $n$-types; for proofs 
see~\cite{TheBook}. Most of the results we cite and mention
were proved by Voevodsky~\cite{voevodsky2013experimental} when he introduced the notion of \emph{homotopy
levels}. He tried to rationalize the numbering by starting at 0, we follow~\cite{TheBook} and the homotopical tradition and start at $-2$. 
The proof that $\typele{n}$ is closed under dependent products requires the 
function extensionality principle, which is itself a consequence of the 
univalence axiom. In fact, one doesn't need to use univalence to show that
function extensionality is equivalent to the principle that the $(-2)$-types
are closed under dependent products. The proof that for $n\geq -1$, $\typele{n}$ is also closed
under the $\mathsf{W}$ type constructor~\cite{awodey2012inductive} is a 
recent result by Danielsson~\cite{Danielsson}.

For the present paper, the results in 
\autoref{ntype-closed-begin} until \autoref{lem:ntype_is_sntype} are of 
particular interest in the case $n\jdeq 0$.

\begin{lem}[See Lemma 4.7.3 and Theorems 4.7.4 and 7.1.4 in \cite{TheBook}]\label{ntype-closed-begin}
A retract of an $n$-truncated type is $n$-truncated. 
Consequently, $n$-truncated types are closed under equivalence.
\end{lem}

\begin{lem}[See Theorem 7.1.9 in \cite{TheBook}]\label{lem:hlevel_prod}
Let $A$ be a type and let $P$ be a family of types over $A$ with
$n$-truncated fibers. Then the dependent function type $\prd{x:A}P(x)$
is $n$-truncated.
\end{lem}

\begin{lem}[See Theorem 7.1.8 in \cite{TheBook}]
Let $A$ be $n$-truncated and $P$ be a family of $n$-truncated
types over $A$. Then the dependent pair type $\sm{x:A}P(x)$
is $n$-truncated.
\end{lem}

\begin{lem}\label{lem:hlevel_w}
Let $A$ be $n$-truncated, $n\geq -1$ and $P:A\to\type$ be a family of types
over $A$. Then the well-ordered type $\wtype{x:A}P(x)$ is $n$-truncated.
\end{lem}

\begin{lem}[See Theorem 7.1.7 in \cite{TheBook}]
If $A$ is $n$-truncated, then $A$ is $(n+1)$-truncated. Hence if $A$ is
$n$-truncated, so is $\id[A]{x}{y}$ for each $x,y:A$.
\end{lem}

The second assertion in the following result requires the univalence axiom.

\begin{lem}[See Theorem 7.1.11 in \cite{TheBook}]\label{lem:ntype_is_sntype}
For any type $A$ and any $n:\Nnegtwo$, 
the type $\istype{n}(A)$ is a mere proposition. The type $\typele{n}$ is itself
an $(n+1)$-truncated type.
\end{lem}

Using higher inductive types, it is possible to implement a left adjoint to the
inclusion $\typele{n}\to\type$. This left adjoint is called the \emph{$n$-truncation}; see~\cite{TheBook}. For
the purpose of this paper, we shall only be concerned with the universal
property for $n$-truncation.

\begin{lem}[A slight extension of Theorem 7.3.2 in \cite{TheBook}, see also
Theorem 7.7.7]
For any type $A$ and any $n:\Nnegtwo$ there is an
$n$-truncated type $\trunc{n}{A}$ and
a function $|-|_n:A\to \trunc{n}{A}$ such that the function
\begin{equation*}
\lam{s}s\circ|-|_n:\prd{w:\trunc{n}{A}}P(w)\to\prd{a:A}P(|a|_n)
\end{equation*}
is an equivalence for every $P:\trunc{n}{A}\to \typele{n}$. 
\end{lem}

The $(-2)$-truncation of a type $A$ is contractible and the $(-1)$-truncation
identifies all elements with each other. Since $(-1)$-truncation is so common, we will often 
omit the subscript and write $\brck{\blank}$ instead of $\trunc{-1}{\blank}$. 
In general, $n$-truncation maps to \typele{n}. Therefore we have for any type $A$ and any $x,y:A$ that the type
$\id[\trunc{n+1}{A}]{|x|_{n+1}}{|y|_{n+1}}$ has all the structure above truncation level $n$ identified.
More precisely, we get:

\begin{lem}[See Theorem 7.3.12 in \cite{TheBook}]\label{trunc-idtypes}
For any type $A$, any $x,y:A$ and any $n:\Nnegtwo$,
there is an equivalence
\begin{equation*}
\eqv{(\id[\trunc{n+1}{A}]{\tproj{n+1}{x}}{\tproj{n+1}{y}})}{\trunc{n}{\id[A]{x}{y}}}
\end{equation*}
\end{lem}

Using the $(-1)$-truncation $\brck{\blank}$, we can fully implement propositions as
$(-1)$-types. 

\begin{defn}
For a type $A$ and (dependent) mere propositions $P$ and $Q$
we define
\begin{align*}
\top       &\defeq\unit         & P \Rightarrow Q      &\defeq P \to Q \\
\bot       &\defeq\emptyt       & \neg P     &\defeq P\to\emptyt  \\
P \land Q  &\defeq P\times Q    & \fall{x : A} P(x)    &\defeq \prd{x : A} P(x) \\
P \lor Q   &\defeq \brck{P + Q} & \exis{x : A} P(x)    &\defeq \sbrck{\sm{x : A} P(x)}\\
\end{align*}
\end{defn}

Recall that bi-implication of $(-1)$-types implies equivalence of those 
$(-1)$-types, which is equivalent by univalence to identity. 
Notice also that although the product of mere
propositions is again a mere proposition, this is not the case for
dependent sums. This is the reason why we needed $(-1)$-truncations to implement
propositions as mere propositions.

\subsection{Surjective and injective functions}
The $n$-truncations are examples of a more general phenomenon called modalities
and a large part of the theory of truncations generalizes to arbitrary
modalities, as we will show in a forthcoming paper. 
A very first approximation is available at~\cite{RijkeSpitters:Factorization}.

\begin{defn}
A function $f:A\to B$ is said to be \emph{$n$-connected} if there is a term of type
\begin{equation*}
\prd{b:B}\iscontr(\strunc{n}{\hfib{f}{B}})
\end{equation*}
\end{defn} 

\begin{defn}
A function $f:A\to B$ is said to be \emph{$n$-truncated} if there is a term of type
\begin{equation*}
\prd{b:B}\istype{n}(\hfib{f}{b})
\end{equation*}
\end{defn}

The main result about the classes $n$-connected and $n$-truncated functions is that they
describe a stable orthogonal factorization system. Every function factors
uniquely as an $n$-connected function followed by an $n$-truncated function, see
Theorem 7.6.6 in \cite{TheBook} for the precise statement. The unique
factorization goes through the $n$-image:

\begin{defn}
Let $f:A\to B$ be a function. We define the \emph{$n$-image of $f$} to be the type
\begin{equation*}
\im_n(f) \defeq \sm{b:B}\strunc{n}{\hfib{f}{b}}
\end{equation*}
\end{defn}

In the present article we are mostly interested in the $(-1)$-connected and the
$(-1)$-truncated maps, which give factorization of functions through their
$(-1)$-image. We will denote the $(-1)$-image of a function $f$ simply by
$\im(f)$ and call it the \emph{image of $f$}. Also, it is more customary to talk
about surjective and injective functions instead of $(-1)$-connected and $(-1)$-truncated
functions. We make the following definitions of surjectivity and injectivity,
which are equivalent to the definitions of $(-1)$-connectedness and
$(-1)$-truncatedness respectively.

\begin{defn}
A function $f:A\to B$ is said to be \emph{surjective} if there is a term of type
\begin{equation*}
\tfsurj(f)\defeq\prd{b:B}\sbrck{\hfib{f}{b}}.
\end{equation*}
\end{defn}

\begin{defn}
A function $f:A\to B$ is said to be \emph{injective} if there is a term of type
\begin{equation*}
\tfinj(f)\defeq\prd{a:A}\iscontr(\hfib{f}{f(a)}).
\end{equation*}
\end{defn}

Below, we give the factorization of any function, but we do not go into the
details of the uniqueness of such a factorization.

\begin{defn}\label{defn:im-fact}
Let $f:A\to B$ be a function. Define
the
functions $\tilde{f}:A\to\im(f)$ and $i_f:\im(f)
\to B$ by
\begin{align*}
\tilde{f} & \defeq  \lam{a}\pairr{f(a),\sbproj{(a,\refl{f(a)})}}\\
i_f & \defeq  \proj1.
\end{align*}
\end{defn}

\begin{lem}
Let $f:A\to B$ be a function. Then $\tilde{f}$ is surjective and $i_f$ 
is injective.
\end{lem}

The injective functions are the monomorphisms of $\set$, which we can also
define via a pullback diagram.

\begin{defn}\label{pb}
Let $f:A\to X$ and $g:B\to X$ be functions. We define the homotopy pullback of $f$ and
$g$ to be
\begin{equation*}
A\times_X B\defeq \sm{a:A}{b:B}(\id[X]{f(a)}{g(b)})
\end{equation*}
and we define $\pi_1:A\times_X B\to A$ and $\pi_2:A\times_X B\to B$ to be the
projections.
\end{defn}

We have the following characterization of $(n+1)$-truncated functions which 
appears in \cite{rezk2010toposes} in the setting of model toposes, 
but not in \cite{TheBook}.

\begin{lem}
A function $f:A\to B$ is $(n+1)$-truncated if and only if the function
\begin{equation*}
\lam{a}\pairr{a,a,\refl{f(a)}}:A\to A\times_B A
\end{equation*}
is $n$-truncated.
\end{lem}

\begin{proof}
Let $\pairr{x,y,p}:\sm{x,y:A}\id{f(x)}{f(y)}$. Then we have the equivalences
\begin{align*}
\hfib{\lam{a}\pairr{a,a,\refl{f(a)}}}{\pairr{x,y,p}} & \eqvsym \sm{a:A}\id{\pairr{a,a,\refl{f(a)}}}{\pairr{x,y,p}}\\
& \eqvsym \sm{a:A}{\alpha:\id{a}{x}}{\beta:\id{a}{y}}\id{\ct{\opp{\ap{f}{\alpha}}}{\ap{f}{\beta}}}{p}\\
& \eqvsym \sm{\beta:\id{x}{y}}\id{\ap{f}{\beta}}{p}\\
& \eqvsym \sm{\beta:\id{x}{y}}\id{\ct{\opp{\ap{f}{\beta}}}{\refl{f(x)}}}{\opp{p}}\\
& \eqvsym \sm{\beta:\id{x}{y}}\id{\trans{\beta}{\refl{f(x)}}}{\opp{p}}\\
& \eqvsym \id[\hfib{f}{f(x)}]{\pairr{x,\refl{f(x)}}}{\pairr{y,\opp{p}}}.
\end{align*}
The latter type is an $n$-type if and only if $\hfib{f}{f(x)}$ is an $(n+1)$-type.
Thus, we see that $\lam{a}\pairr{a,a,\refl{f(a)}}$ is $n$-truncated if and only if
$\hfib{f}{f(x)}$ is an $(n+1)$-type for each $x:A$. Note that each $\hfib{f}{f(x)}$
is an $(n+1)$-type if and only if each $\hfib{f}{b}$ is an $(n+1)$-type.
\end{proof}

In particular, a function is injective if and only if the function $A\to
A\times_B A$ is an equivalence. In other words, a function is injective 
precisely when it is a monomorphism, i.e.~when the diagram
\begin{equation*}
\begin{tikzcd}
A \rar{\idfunc[A]} \dar[swap]{\idfunc[A]} & A \dar{f} \\
A \rar[swap]{f} & B
\end{tikzcd}
\end{equation*}
is a pullback diagram.

Let us make the
 verifications of two of the ingredients of a
predicative topos. The first is that sums are disjoint. 

\begin{lem}[See Theorem 2.12.5 in \cite{TheBook}]\label{lem:sumsdisjount}
For any two types $X$ and $Y$ we have the equivalences
\begin{align*}
(\id[X+Y]{\inl(x)}{\inl(x')}) & \eqv{}{(\id[X]{x}{x'})}\\
(\id[X+Y]{\inl(x)}{\inr(y)}) & \eqv{}{\emptyt}\\
(\id[X+Y]{\inr(y)}{\inr(y')}) & \eqv{}{(\id[Y]{y}{y'})}
\end{align*}
Consequently, the inclusions $\inl:X\to X+Y$ and 
$\inr:Y\to X+Y$ are monomorphisms and the diagram
\begin{equation*}
\begin{tikzcd}
\emptyt \rar \dar & X \dar{\inl} \\
Y \rar[swap]{\inr} & X + Y
\end{tikzcd}
\end{equation*}
is a pullback diagram.
\end{lem}

We also have the following general result, which has the consequence that
$\set $ is \emph{lextensive}.

\begin{thm}\label{prop:sums_stable}
Let $P:A\to\type$ be a family of types and let $f:(\sm{a:A}P(a))\to
B$ and $g:X\to B$ be functions.
Then there is an equivalence 
\begin{equation*}
\eqv{\big(\sm{a:A}(P(a)\times_B X)\big)}{\big(\sm{a:A}P(a)\big)\times_B X}
\end{equation*}
In particular, there is an equivalence
\begin{equation*}
\eqv{(A_0\times_B X)+(A_1\times_B X)}{(A_0+A_1)\times_B X}
\end{equation*}
for any three functions $f_0:A_0\to B$ and $f_1:A_1\to B$ and $g:X\to B$.
\end{thm}

\begin{proof}
Note that we have the equivalences
\begin{align*}
\big(\sm{a:A}(P(a)\times_B X)\big) 
& \jdeq \sm{a:A}{u:P(a)}{x:X}(\id[B]{f(a,u)}{g(x)})\\
& \eqv{}{\sm{w:\sm{a:A}P(a)}{x:X}(\id[B]{f(w)}{g(x)})}\\
& \jdeq (\sm{a:A}P(a))\times_B X\qedhere
\end{align*}
\end{proof}

\subsection{Homotopy colimits}
In contrast with limits, 
general colimits such as quotients are not provided by Martin-L\"of type 
theory. Thus the category of setoids -- the left exact completion of the category
of types -- was considered to work around this deficit. In the univalent foundations,
quotients can be introduced as 
higher inductive types. We present only the results that are essential in
the context of sets, a more thorough discussion
about higher inductive types can 
be found either in chapter 6 of \cite{TheBook} or in
\cite{RijkeSpitters:colims}.

\begin{defn}
A \emph{(directed) graph} $\Gamma$ is a pair $\pairr{\pts{\Gamma},\edg{\Gamma}}$
consisting of a type $\pts{\Gamma}$ of \emph{points} and a binary relation $\edg{\Gamma}:\pts{\Gamma}\to\pts{\Gamma}\to \prop $
of \emph{edges}.
\end{defn}

\begin{defn}\label{colim}
Let $\Gamma$ be a graph. We define $\tfcolim(\Gamma)$ to be the higher inductive
type with basic constructors
\begin{align*}
\pts{\alpha} & : \pts{\Gamma}\to\tfcolim(\Gamma)\\
\edg{\alpha} & : \prd{i,j:\pts{\Gamma}}\edg{\Gamma}(i,j)\to\id{\pts{\alpha}(i)}{\pts{\alpha}(j)}
\end{align*}
The induction principle for $\tfcolim(\Gamma)$ is that for any family $P:\tfcolim(\Gamma)\to\type$,
if there are
\begin{align*}
\pts{H} & : \prd{i:\pts{\Gamma}}P(\pts{\alpha}(i))\\
\edg{H} & : \prd{i,j:\pts{\Gamma}}{q:\edg{\Gamma}(i,j)}\id{\trans{\edg{\alpha}(q)}{\pts{H}(i)}}{\pts{H}(j)}
\end{align*}
then there is a dependent function $f:\prd{w:\tfcolim(\Gamma)}P(w)$ with
\begin{align*}
f(\pts{\alpha}(i)) & \defeq \pts{H}(i) & & \text{for }i:\pts{\Gamma}\\
\apd{f}{\edg{\alpha}(q)} & := \edg{H}(i,j,q) & & \text{for }i,j:\pts{\Gamma}\text{ and }q:\edg{\Gamma}(i,j).
\end{align*}
\end{defn}

These colimits are higher inductive types of the kind that are presented in section
6.12 in \cite{TheBook}, using a binary relation over the type $\pts{\Gamma}$ rather
than a pair of functions into $\pts{\Gamma}$.

Using these higher inductive types, all the homotopy colimits that appear in \cite{TheBook} can be constructed.
In particular, we have the pushout of $f:A\to B$ and $g:A\to C$ by considering the
graph $\Gamma$ with
\begin{align*}
\pts{\Gamma} & \defeq B+C\\
\edg{\Gamma}(i,j) & \defeq \hfib{\varphi}{\pairr{i,j}}
\end{align*}
where we take
\begin{equation*}
\varphi \defeq \lam{a}\pairr{\inl(f(a)),\inr(g(a))} : A\to (B+C)^2.
\end{equation*}
To see what $\hfib{\varphi}{\blank}$ is, note that 
\begin{equation*}
\eqv
  {\hfib{\varphi}{\pairr{\inl(b),\inr(c)}}}
  {\sm{a:A}(\id{f(a)}{b})\times(\id{g(a)}{c})}
\end{equation*}
and that $\hfib{\varphi}{\blank}$ is empty for other combinations of $\inl$ and $\inr$.
The pushout is equivalently described with the basic constructors
\begin{align*}
\inl & : B\to B+_A C\\
\inr & : C\to B+_A C\\
\glue & : \prd{a:A}\id{\inl(f(a))}{\inr(g(a))}.
\end{align*}

We note that using higher inductive types and univalence for propositions, 
it is possible to give
a new proof of the fact that the axiom of choice, see Eq.~3.8.1 in \cite{TheBook},
implies the law of excluded middle, see Eq.~3.4.1 in \cite{TheBook}. Although
we will not use this fact in the proof that $\set$ forms a predicative topos,
the higher inductive type is an instance of the construction of quotients.

\begin{defn}
Suppose $P$ is a mere proposition. We define the auxilary binary relation $R_P:\bool
\to\bool\to\type$ by $R_P(\bfalse,\btrue)\defeq P$ and $R_P(b,b')\defeq\emptyt$ otherwise.
Define $\bool/P$ to be the type
\begin{equation*}
\tfcolim(\pairr{\bool,R_P}).
\end{equation*}
\end{defn}

Using \autoref{thm:h-set-refrel-in-paths-sets} it is not hard to see that
$\bool/P$ is a set, see Lemma 10.1.13 in \cite{TheBook}. The basic constructor
$\bool\to\bool/P$ is a surjective function, so we may use the axiom of choice
to obtain a section and use the decidability of equality in $\bool$ to
decide whether $P$ or $\neg P$ holds:

\begin{thm}[See Theorem 10.1.14 in \cite{TheBook}]\label{prop:1surj_to_surj_to_pem}
If all surjections between sets merely split, then the law of excluded middle follows.
\end{thm}

Because truncations are left-adjoints, we note that when $\Gamma$ is a graph,
then $\trunc{0}{\tfcolim(\Gamma)}$ is the set-colimit of $\Gamma$. To see this,
note that the truncation $\trunc{0}{\blank}$ restricts the universal property
of $\tfcolim(\Gamma)$ to only those cases where a family of sets over $\tfcolim(\Gamma)$
is considered.

We end the preliminaries with a discussion on how to take the coequalizer of
two functions $f,g:A\to B$ in the univalent category of sets. Its definition
as a higher inductive type is straightforward:

\begin{defn}
Let $f,g:A\to B$ be functions between sets. We define the set-coequalizer of $f$
and $g$ to be the type
$\coeq{B}{f}{g}\defeq\strunc{0}{\tfcolim(\pairr{B,\eq{f}{g}})}$, where $\eq{f}{g}$ is the family
defined by
\begin{equation*}
\eq{f}{g}(b,b')\defeq \sm{a:A}(\id{f(a)}{b})\times(\id{g(a)}{b'})
\end{equation*}
We denote the composite function $B\to \tfcolim(\pairr{B,\eq{f}{g}})\to \coeq{B}{f}{g}$
by $c_{f,g}$. A
\emph{regular epimorphism} is a function between sets which is (homotopic to) a set-coequalizer.
\end{defn}

The following lemma explains that the coequalizer of $f$ and $g$ has indeed the
right universal property. This is a general phenomenon; 
truncated colimits behave as expected.

\begin{lem}[See Lemma 10.1.4 in \cite{TheBook}]\label{coeq}
Let $f,g:A\to B$ be functions between sets $A$ and $B$. The 
{set-co}equalizer $c_{f,g}:B\to \coeq{B}{f}{g}$ satisfies the universal property
\begin{equation*}
\prd{C:\set}{h:B\to C}{H:h\circ f\htpy h\circ g}
\iscontr(\sm{k:\coeq{B}{f}{g}\to C}k\circ c_{f,g}\htpy h).
\end{equation*}
\end{lem}

\section{$\catset$ is a \texorpdfstring{$\Pi$}{Π}$\mathsf{W}$-pretopos}\label{sec:quotient}
In this section we begin by verifying the principle of unique choice.
The importance of this result is not in the difficulty of its proof, 
but in the absence of the result in some other type theories.
In these type theories one introduces a separate sort of `propositions', 
which, however, are not necessarily identified with mere 
propositions. Such an approach may be more general, but less powerful. 
The principle of unique choice fails in the calculus of constructions~\cite{streicher1992independence}, in logic enriched type 
theory~\cite{aczel2002collection}, in minimal type 
theory~\cite{maietti2005toward} and in the category of prop-valued setoids 
in \Coq~\cite{Spiwack}. 
This principle does hold the model of total setoids using a propositions as types interpretation~\cite{hofmann1995extensional}.

We will show that in the presence of $(-1)$-truncation, $\catset$ becomes
a regular category. The natural candidate for coequalizer of the kernel pair 
of a function is the image of the function. Our proof that the image is indeed
the coequalizer is an application of the principle of unique choice. 
This work is reminiscent of the connections between 
$[\ ]$-types~\cite{AwodeyBauer2004}, or mono-types~\cite{MSC:359803}, and regular 
categories in an extensional setting.

To show that $\catset$ is exact provided that we have quotients, 
we need to show in addition that every equivalence relation is effective. 
In other words, given an equivalence
relation $R:A\to A\to\prop $, there is a coequalizer $c_R$ of the pair
$\pi_1,\pi_2:(\sm{x,y:A}R(x,y))\to A$.

We consider the pre-category $\eqrel$, which becomes a $1$-category by univalence.
The pre-category $\eqrel$ shares many properties of the pre-category $\mathsf{Std}$
of setoids, which is the exact completion of $\catset$. Using higher inductive
types and univalence, we will show that we have
a quotient functor $\eqrel \to\catset$ which is moreover left
adjoint to the inclusion $\catset\to\eqrel$.
This adjunction is in general not an equivalence; proving this requires 
the axiom of choice. With $\catset$ being exact, we will be ready to show that
$\catset$ forms a $\piw$-pretopos.

After having shown that $\catset$ is exact, we will show that $\type$ has an object
classifier. From this we will derive that $\catset$ has a subobject classifier.
This also shows that if we the resizing rule that $\prop$ is equivalent to
a type in $\set$, then $\catset$ actually becomes a topos.

\subsection{Regularity of the category of sets}
\begin{defn}
Suppose $P:A\to\type$ is a family of types over $A$. We define
\begin{align*}
\atMostOne(P) & \defeq   \prd{x,y:A}P(x)\to P(y)\to (\id{x}{y})\\
\uexis{x:A}P(x) & \defeq   (\exis{x:A}P(x))\times\atMostOne(P).
\end{align*}
\end{defn}

\begin{lem}\label{lem:atMostOne}
Suppose that $P:A\to\prop $. If there is an element $H:\atMostOne(P)$, then the type $\sm{x:A}P(x)$ is a mere proposition.
\end{lem}

\begin{proof}
Suppose that $\pairr{x,u}$ and $\pairr{x',u'}$ are elements of $\sm{x:A}P(x)$. Then we have the path $p\defeq
H(u,u'):\id{x}{x'}$. Moreover, there is a path from $\trans{p}{u}$ to $u'$ since $P(x')$ is assumed to be a
mere proposition.
\end{proof}

\begin{lem}\label{lem:iota}
For any family $P:A\to\prop $ of mere propositions there is a function of type
\begin{equation*}
(\uexis{x:A}P(x))\to\sm{x:A}P(x).
\end{equation*}
\end{lem}
\begin{proof}
Suppose we have $H:\atMostOne(P)$ and $K:\exis{x:A}P(x)$. From $H$ it follows that $\sm{x:A}P(x)$ is a mere proposition, and
therefore it follows that $\eqv{(\sm{x:A}P(x))}{(\exis{x:A}\,P(x))}$. 
\end{proof}

\begin{thm}[The principle of unique choice]\label{thm:UAC}
Suppose that $A$ is a type, that $P:A\to\type$ is a family of types over $A$ 
and that $R:\prd{x:A}(P(x)\to\prop )$ is a family of mere propositions over 
$P$. Then there is a function
\begin{align*}
{\textstyle(\prd{x:A}\uexis{u:P(x)}R(x,u))\to\sm{f:\prd{x:A}P(x)}\prd{x:A}R(x,f(x))}.
\end{align*}
\end{thm}

\begin{proof}
Suppose that $H:\prd{x:A}(\uexis{u:P(x)}\,R(x))$. 
By \autoref{lem:iota} we can find an element of type $\sm{u:P(x)}R(x,u)$ 
for every $x:A$. A function
\begin{align*}
{\textstyle(\prd{x:A}\sm{u:P(x)}R(x,u))\to(\sm{f:\prd{x:A}P(x)}\prd{x:A}R(x,f(x)))}
\end{align*}
is obtained from the usual choice principle, called $\choice{\infty}$ in \cite{TheBook}.
\end{proof}

The following seemingly stronger variant of $\atMostOne(P)$ helps us
showing that $\atMostOne(P)$ is a mere proposition for every $P:A\to
\prop $. 

\begin{defn}
Let $P:A\to\type$ be a family of types over a type $A$. We define
\begin{equation*}
\mathsf{baseLevel}(-1,P) 
\defeq  \prd{x,y:A}P(x)\to P(y)\to\iscontr (\id{x}{y}).
\end{equation*}
\end{defn}

Notice that
we could replace $\iscontr $ in the definition of $\mathsf{baseLevel}(-1)$
by $\istype{(-2)}{}$ and see that we can easily generalize the notion of
$\mathsf{baseLevel}(-1)$ to $\mathsf{baseLevel}(n)$ for $n\geq -1$.

\begin{lem}\label{lem:atmostone_to_baseLevel}
For any $P:A\to\prop $, there is a function of type
\begin{equation*}
\atMostOne(P)\to\mathsf{baseLevel}(-1,P).
\end{equation*}
\end{lem}

\begin{proof}
We will show that there is a function of type
\begin{equation*}
\atMostOne(P)\to\prd{x,y:A}P(x)\to P(y)\to\iscontr (\id{x}{y}).
\end{equation*}
Let $H$ be an element of type $\atMostOne(P)$ and let $x,y:A$, $u:P(x)$ 
and $v:P(y)$. Then we have the terms $\pairr{x,u}$ and $\pairr{y,v}$ in
$\sm{x:A}P(x)$. Since $\sm{x:A}P(x)$ is a proposition, the path space
$\id{\pairr{x,u}}{\pairr{y,v}}$ is contractible. Since $P$ is assumed to
be a proposition, there is an equivalence $\eqv{(\id{x}{y})}{(\id{\pairr{x,u}}{\pairr{y,v}})}$.
Hence it follows that $\id{x}{y}$ is contractible.
\end{proof}

\begin{cor}
For any family $P:A\to\prop$ of mere propositions, $\atMostOne(P)$ is a mere 
proposition equivalent to $\mathsf{baseLevel}(-1,P)$.
\end{cor}

\begin{proof}
By \autoref{lem:atmostone_to_baseLevel} any two proofs
of $\atMostOne(P)$ are homotopic, as the path spaces $\id{x}{y}$ are
contractibile for $x,y:A$ such that $P(x)$ and $P(y)$. Thus, $\atMostOne(P)$ is a mere proposition. Since there is a
bi-implication between the mere propositions $\atMostOne(P)$ and 
$\mathsf{baseLevel}(-1,P)$, they are equivalent.
\end{proof}

As an application of unique choice, we show that surjective
functions between sets are regular epimorphisms.

\begin{defn}
Let $f:A\to B$ be a function between sets. Define
\begin{align*}
\isepi{f} & \defeq  \prd{X:\set}{g,h:B\to X}
(g\circ f\sim h\circ f)\to (g\sim h).
\end{align*}
\end{defn}

Since we have restricted the condition of being an epimorphism to the category
of sets, the type $\isepi{f}$ is a mere proposition.

\begin{lem}\label{epis-surj}
Let $f$ be a function between sets. The following are equivalent:
\begin{enumerate}
\item $f$ is an epimorphism.
\item Consider the pushout diagram
\begin{equation*}
\begin{tikzcd}
A \rar{f} \dar & B \dar{\iota} \\ \unit  \rar[swap]{t} & \mappingcone{f}
\end{tikzcd}
\end{equation*}
defining the mapping cone of $f$. The type $\strunc{0}{\mappingcone{f}}$ is contractible.
\item $f$ is surjective.
\end{enumerate}
\end{lem}

\begin{proof}
To show that $\isepi{f}\to\iscontr (\strunc{0}{\mappingcone{f}})$, suppose that
$H:\isepi{f}$. The basic constructor $t$ of $\mappingcone{f}$ gives us the element 
$\tproj{0}{t(\ttt)}:\strunc{0}{\mappingcone{f}}$. We have to show that
\begin{equation*}
\prd{x:\strunc{0}{\mappingcone{f}}}\id{x}{\tproj{0}{t(\ttt)}}.
\end{equation*}
Note that the type $\id{x}{\tproj{0}{t(\ttt)}}$ is a mere proposition because 
$\strunc{0}{\mappingcone{f}}$ is a set, hence
it is equivalent to show that
\begin{equation*}
\prd{w:\mappingcone{f}}\id{\tproj{0}{w}}{\tproj{0}{t(\ttt)}}
\end{equation*} 
which is by \autoref{trunc-idtypes} equivalent to
\begin{equation*}
\prd{w:\mappingcone{f}}\sbrck{\id{w}{t(\ttt)}}.
\end{equation*} 
We can use induction on
$\mappingcone{f}$: it suffices to find
\begin{align*}
I_0 & : \prd{b:B}\sbrck{\id{\iota(b)}{t(\ttt)}}\\
I_1 & : \prd{a:A}\id{\trans{\glue(a)}{I_0(f(a))}}{\sbproj{\refl{t(\ttt)}}}.
\end{align*}
where $\glue:\prd{a:A}\id{\iota(f(a))}{t(\ttt)}$ is the path constructor
of $\mappingcone{f}$.  Since the type of $I_1$ is the type of sections of a family
of identity types of propositions -- which are thus contractible -- we get $I_1$ for free. 
Since $f$ is epi and since we have $\glue:\iota\circ f \sim (\lam{b}t(\ttt))\circ f$, we get a homotopy  $\iota \sim 
\lam{b}t(\ttt)$, which gives us $I_0$.

To show that $\iscontr (\strunc{0}{\mappingcone{f}})\to\mathsf{surj}(f)$,
let $H:\iscontr (\strunc{0}{\mappingcone{f}})$. Using the univalence axiom, we construct
a family $P:\strunc{0}{\mappingcone{f}}\to\prop$ of mere propositions. Note that $\prop$ is
a set, so it suffices to define the family $P(\strunc{0}{\blank}):\mappingcone{f}\to\prop$.
For this we can use induction on $\mappingcone{f}$: we define
\begin{align*}
P(\tproj{0}{t(x)}) & \defeq  \unit & & \text{for }x:\unit\\
P(\tproj{0}{\iota(b)}) & \defeq  \sbrck{\hfib{f}b} & & \text{for }b:B.
\end{align*}
For $a:A$ the type $\sbrck{\hfib{f}{f(a)}}$ is canonically
equivalent to $\unit $, which finishes the construction of $P$.
Since $\strunc{0}{\mappingcone{f}}$ is assumed to be contractible it follows that $P(x)$ is
equivalent to $P(\tproj{0}{t(\ttt)})$ for any $x:\strunc{0}{\mappingcone{f}}$. In particular we find that
$\sbrck{\hfib{f}b}$ is contractible for each $b:B$, showing
that $f$ is surjective.

To show that $\mathsf{surj}(f)\to\isepi{f}$,
let $f:A\to B$ be a surjective function and consider a set $C$ and two functions
$g,h:B\to C$ with the property that $g\circ f\sim h\circ f$. Since $f$ 
is assumed to be surjective,
we have an equivalence $\eqv{B}{\im(f)}$. Since identity types in sets are
propositions, we get
\begin{align*}
\prd{b:B}\id{g(b)}{h(b)}
& \eqvsym \prd{w:\im(f)}\id{g(\proj1 w)}{h(\proj1(w))}\\
& \eqvsym \prd{b:B}{a:A}{p:\id{f(a)}{b}}\id{g(b)}{h(b)}\\
& \eqvsym \prd{a:A}\id{g(f(a))}{h(f(a))}.
\end{align*}
By assumption, there is an element of the latter type.
\end{proof}

The proof that epis are surjective in~\cite{Mines/R/R:1988} uses the power set operation. 
This proof can be made predicative by using a large power set and typical ambiguity.
A predicative proof for setoids was given by Wilander~\cite{Wilander2010}. 
The proof above is similar, but avoids setoids by using pushouts and the
univalence axiom.

\begin{thm}\label{prop:images_are_coequalizers}
Surjective functions between sets are regular epimorphisms.
\end{thm}

\begin{proof}
Note that it suffices to show that for any function $f:A\to B$, the diagram
\begin{equation*}
\begin{tikzcd}
\sm{x,y:A}\id{f(x)}{f(y)} \ar[yshift=.7ex]{r}{\pi_1} 
\ar[yshift=-.7ex]{r}[swap]{\pi_2} & A \ar{r}{\tilde{f}} & \im(f)
\end{tikzcd}
\end{equation*}
is a coequalizer diagram (recall the definition of $\tilde{f}$ from
\autoref{defn:im-fact}).

We first construct a homotopy $H:\tilde{f}\circ\pi_1\sim\tilde{f}\circ\pi_2$. Let $\pairr{x,y,p}$ be an element
of $\sm{x,y:A}\id{f(x)}{f(y)}$. 
Then we have $\id{\tilde{f}(\pi_1(\pairr{x,y,p}))}{\pairr{f(x),u}}$, 
where $u$ is an element of the contractible type $\sbrck{\hfib{f}{f(x)}}$. 
Similarly, we have a path 
$\id{\tilde{f}(\pi_2(\pairr{x,y,p}))}{\pairr{f(y),v}}$,
where $v$ is an element of the contractible type $\sbrck{\hfib{f}{f(y)}}$.
Since we have $p:\id{f(x)}{f(y)}$ and since
$\sbrck{\hfib{f}{f(y)}}$ is contractible, 
it follows that we get a path from $\tilde{f}(\pi_1(\pairr{x,y,p}))$ to
$\tilde{f}(\pi_2(\pairr{x,y,p}))$, which gives us our homotopy $H$.

Now suppose that $g:A\to X$ is a function for which there is a homotopy 
$K:g\circ\pi_1\sim g\circ\pi_2$. We have to show that the type
\begin{equation*}
\sm{h:\im(f)\to X} h\circ\tilde{f}\sim g
\end{equation*}
is contractible. We will apply unique choice to define a 
function from $\im(f)$ to $X$. Let $R:\im(f)\to
X\to\prop$ be the relation defined by 
\begin{align*}
R(w,x) & \defeq \prd{a:A}(\id{\tilde{f}(a)}{w})\to (\id{g(a)}{x})
\end{align*}
There is an element of $\atMostOne(R(w))$ for every $w:\im(f)$. 
To see this, note that the type $\atMostOne(R(w))$ is a
mere proposition. Therefore, there is an equivalence
\begin{equation*}
\eqv{\big(\prd{w:\im(f)}\atMostOne(R(w))\big)}
{\prd{a:A}\atMostOne(R(\tilde{f}(a)))}.
\end{equation*}
Let $a:A$, $x,x':X$,
$u:R(\tilde{f}(a),x)$ and $u':R(\tilde{f}(a),x')$. 
Then there are the paths $u(a,\refl{\tilde{f}(a)}):\id{g(a)}{x}$
and $u'(a,\refl{\tilde{f}(a)}):\id{g(a)}{x'}$, 
showing that $\id{x}{x'}$. 

Also, there is an element of $\exis{x:X}R(w,x)$ for every $w:\im(f)$. Indeed, the type
\begin{equation*}
\prd{w:\im(f)}\exis{x:X}R(w,x)
\end{equation*}
is equivalent to the type
\begin{equation*}
\prd{a:A}\exis{x:X}R(\tilde{f}(a),x).
\end{equation*}
The type $\prd{a:A}\sm{x:X}R(\tilde{f}(a),x)$ is inhabited by the element
\begin{equation*}
\lam{a}\pairr{g(a),(\lam{a'}{p}K(\pairr{a',a,\opp{p}}))}
\end{equation*}
This shows that the hypotheses of the principle of unique choice are satisfied, so we get an element of type
\begin{equation*}
\sm{h:\im(f)\to X}\prd{w:\im(f)}R(w,h(w)).
\end{equation*}
An immediate consequence of the way we constructed our function $h:\im(f)\to X$ is that $h\circ\tilde{f}\sim g$. The result follows
now from the observation that the type
\begin{equation*}
\sm{h':\im(f)\to X}h'\circ\tilde{f}\sim h\circ\tilde{f}
\end{equation*}
is contractible because $\tilde{f}$ is an epimorphism. 
\end{proof}

\begin{lem}\label{lem:pb_of_coeq_is_coeq}
Pullbacks of surjective functions are surjective. Consequently,
pullbacks of coequalizers are coequalizers.
\end{lem}

\begin{proof}
Consider a pullback diagram
\begin{equation*}
\begin{tikzcd}
A \rar \dar[][swap]{f} & B \dar{g} \\ C \rar[][swap]{h} & D
\end{tikzcd}
\end{equation*}
and assume that $g$ is surjective. Applying the pasting lemma of pullbacks
with the morphism $c:\unit \to C$, we obtain an
equivalence $\eqv{\hfib{f}c}{\hfib{g}{h(c)}}$ for any
$c:C$. This equivalence gives that $f$ is surjective.
\end{proof}

\begin{thm}\label{thm:set_regular}
The category $\catset$ is regular.
\end{thm}

\begin{proof}
$\catset$ has all limits, so it is finitely complete. 
\autoref{prop:images_are_coequalizers} gives
that the kernel pair of each function has a coequalizer.
\autoref{lem:pb_of_coeq_is_coeq} gives that
coequalizers are stable under pullbacks.
\end{proof}

\subsection{The $1$-category of equivalence relations}

Setoids were introduced by Bishop~\cite{Bishop1967} to model extensional 
functions in an unspecified effective framework.
Hofmann~\cite{hofmann1995extensional} developed this theory to build a model of extensional type theory 
in an intensional type theory.
However, with general setoids we cannot hope to obtain a pre-category in
the sense of \cite{TheBook} and we do wish to end up with a (Rezk-complete) 
$1$-category when we do assume
univalence. Hence we shall 
restrict to \emph{mere propositional} equivalence relations over sets to obtain the 
pre-category of setoids. These are the objects of the precategory $\eqrel$;
the morphisms will be required to respect these
mere equivalence relations.

As a consequence of this restriction, $\catset$ becomes a reflective subcategory $\eqrel$. 
We establish this result in \autoref{sec:quotients}. 
In the presence of the axiom of choice \cite[Eq.~3.8.1]{TheBook} 
the categories are even equivalent. 

The pre-category $\eqrel$ is reminiscent of the ex/lex completion of $\catset$.
Note that instead of functions one could also consider total functional relations. 
Since $\catset$ is a regular $1$-category satisfying the principle of unique 
choice (\autoref{thm:UAC}), these options are equivalent.
The construction in \autoref{sec:resizing} uses a Yoneda construction which is
also reminiscent of the ex/lex completion; see~\cite{maietti2012elementary} for an overview.

\begin{defn}
An \emph{equivalence relation} over a type $A$ consists of a mere relation
$R :A\to A\to\prop $ which is \emph{reflexive}, \emph{symmetric} and
\emph{transitive}, i.e.\ there are elements
\begin{align*}
\rho & : \prd{x:A}R (x,x)\\
\sigma & : \prd{x,y:A}R (x,y)\to R (y,x)\\
\tau & : \prd{x,y,z:A}R (y,z)\to R (x,y)\to R (x,z).
\end{align*}
We also write $\mathsf{isEqRel}(R)$ for the type witnessing that
$R $ is an equivalence relation.
\end{defn}

\begin{defn}
We define
\begin{equation*}
\mathsf{ob}(\eqrel )\defeq \sm{A:\set}{R:A\to A\to\prop }
\mathsf{isEqRel}( R).
\end{equation*}
Usually we shall slightly abuse notation and speak of $\pairr{A,R}$ as an object
of $\eqrel $, leaving the witness of $\mathsf{isEqRel}(R)$ implicit.
\end{defn}

\begin{defn}
Let $\Gamma\defeq \pairr{A,R}$ and $\Delta\defeq \pairr{B,S}$
be objects of $\eqrel $. A morphism $f$ from $\Delta$ to $\Gamma$ is a
pair $\pairr{f_0,f_1}$ consisting of
\begin{align*}
f_0 & : B\to A\\
f_1 & : \prd{x,y:B}S(x,y)\to R(f_0(x),f_0(y)).
\end{align*}
Thus, we define
\begin{equation*}
\mathsf{hom}(\Delta,\Gamma) \defeq  \sm{f_0:B\to A}
\prd{x,y:B}S(x,y)\to R(f_0(x),f_0(y))
\end{equation*}
which is a sub\emph{set} of the function \emph{set}.
We will usually denote the type $\mathsf{hom}(\Delta,\Gamma)$ by $\Delta\to
\Gamma$. The identity morphisms $\idfunc[\Gamma]$ and the composite morphisms
$g\circ f$ are defined in the obvious way.
\end{defn}

In the following theorem we use the univalence axiom to deduce that we get
a $1$-category of equivalence relations.
However, in our result that $\catset$ is a $\piw$-pretopos we will not
use univalence and hence we will not use the fact that $\eqrel $ is
a $1$-category.

\begin{lem}
For any two setoids $\Gamma$ and $\Delta$, the type
of isomorphisms from $\Delta$ to $\Gamma$ is equivalent to the type $\id{\Delta}{\Gamma}$ in 
$\mathsf{ob}(\eqrel )$. 
\end{lem}
\begin{proof}
First observe that we have an equivalence
\begin{equation*}
\eqv{(\id{\Delta}{\Gamma})}{\big(\sm{e:\eqv{B}{A}}\prd{x,y:A}
\eqv{S(e^{-1}(x),e^{-1}(y))}{R(x,y)}\big)}.
\end{equation*}
As for the type of isomorphisms from $\Delta$ to $\Gamma$, note that we have an
equivalence
\begin{align*}
\eqv{(\Delta\cong\Gamma)}{} 
& (\sm{\pairr{f_0,g_0,\eta_0,\varepsilon_0}:B\cong A}\\
& \qquad (\prd{x,y:B}S(x,y)\to R(f_0(x),f_0(y)))\\
& \qquad \times(\prd{x,y:A}R(x,y)\to S(g_0(x),g_0(y)))).
\end{align*}
Since $B$ and $A$ are assumed to be sets, we have that
$\eqv{(B\cong A)}{(\eqv{B}{A})}$ and therefore it suffices
to show that the type
\begin{align*}
{\textstyle(\prd{x,y:B}S(x,y)\to R(f_0(x),f_0(y)))
\times (\prd{x,y:A}R(x,y)\to S(g_0(x),g_0(y)))}
\end{align*}
is equivalent to the type $\prd{x,y:A}
\eqv{(S(g_0(x),g_0(y))}{R(x,y))}$ for every isomorphism 
$\pairr{f_0,g_0,\eta_0,\epsilon_0}:{B\cong A}$.
Note that both types are mere propositions,
so we only have to find implications in both directions. These can be found by
using $\eta_0:g_0\circ f_0\sim\idfunc[B]$ and $\epsilon_0:f_0\circ g_0
\sim\idfunc[A]$.
\end{proof}

We already mentioned the inclusion of $\catset$ into $\eqrel $. We 
define it on objects by $A\mapsto\pairr{A,\idtypevar{A}}$. Recall that for each function
$f:A\to B$ there is a function $\mapfunc{f}:
\prd{x,y:A}(\id{x}{y})\to (\id{f(x)}{f(y)})$, and hence defines a map
between the setoids $\pairr{A,\idtypevar{A}}$ and $\pairr{B,\idtypevar{B}}$.
It comes at no surprise that this determines a functor R. 

\begin{lem}
The inclusion of $\catset$ into $\eqrel $ is full and faithful.
\end{lem}

\begin{proof}
We have to show that for any two sets $A$ and $B$, the inclusion determines
an equivalence $\eqv{(A\to B)}{(\pairr{A,\idtypevar{A}}\to\pairr{B,\idtypevar{B}})}$.
Naturally, we choose the inverse to be the first projection. Since the identity
types on sets are mere propositions, it is immediate that the first projection is
a section for the inclusion.
\end{proof}

\subsection{Quotients}\label{sec:quotients}
Given an object $\pairr{A,R}$ of $\eqrel $, we wish to define a set $Q(A,R)$
together with a morphism $f:\pairr{A,R}\to\pairr{Q(A,R)
,\idtypevar{Q(A,R)}}$ in $\eqrel $
with the universal property that precomposition with $f_0$ gives an equivalence
\begin{equation*}
\eqv{(Q(A,R)\to Y)}{\mathsf{hom}_\eqrel (\pairr{A,R },\pairr{Y,\idtypevar{Y}})}
\end{equation*}
for every set $Y$. In other words, we are looking for a left adjoint to the
inclusion $X\mapsto\pairr{X,\idtypevar{X}}$ of $\catset$ into $\eqrel $.
Such a left adjoint is mapping the setoid $\pairr{A,R}$ to
the quotient $A/R $.

There are several solutions to this problem, of which we present two. 
The first solution uses higher inductive types of the kind presented in
\autoref{colim}.

\begin{defn}
Let $A$ be a set and let $R:A\to A\to\prop $ be a binary mere relation over
$A$ (not necessarily an equivalence relation). 
We define $A/R$ to be the type $\trunc{0}{\tfcolim(\pairr{A,R})}$.
\end{defn}

Since a binary relation $R:A\to A\to\type$ is equivalently described as a pair
of functions by the two projections $\pi_1,\pi_2:(\sm{x,y:A}R(x,y))\to A$, we
get the following lemma from \autoref{coeq}:

\begin{lem}
Let $A$ be a set and let $R:A\to A\to\prop$ be a binary mere relation over 
$A$. Then $A/R$ is the (set-)coequalizer of the two projections
$\pi_1,\pi_2:(\sm{x,y:A}R(x,y))\to A$.
\end{lem}

Using the induction principle of each $A/R$, we can extend the function
$\lam{\pairr{A,R}}A/R$ to a functor $Q$ from $\eqrel $ to $\catset $
in a canonical way. 

We check that quotients have the expected universal properties.

\begin{thm}\label{thm:set_EqRel}
The functor $Q$ is left adjoint to the inclusion 
$i:\catset \to \eqrel $. Thus, $\catset $ is a reflective
subcategory of $\eqrel $.
\end{thm}

\begin{proof}
We have to show that there are
\begin{enumerate}
\item a unit $\eta:\mathbf{1}\to i\circ Q$.
\item and a counit $\varepsilon:Q\circ i\to\mathbf{1}$
\item satisfying the triangle identities 
\begin{equation*}
\id{\varepsilon_{A/R}\circ Q(\eta_{(A,R)})}{\refl{A/R}}
\qquad\text{and}\qquad
\id{i(\varepsilon_A)\circ \eta_{\pairr{A,\idtypevar{A}}}}{\refl{\pairr{A,\idtypevar{A}}}}.
\end{equation*}
\end{enumerate}
For the unit we take $\eta_{(A,R),\,0}\defeq  c_R$ and $\eta_{(A,R),\,1}\defeq p_R$, 
where $c_R$ is the coequalizer of
the pair $\pi_1,\pi_2:(\sm{x,y:A}R(x,y))\to A$ and where 
\begin{equation*}
p_R:\prd{x,y:A}R(x,y)\to (\id{c_R(x)}{c_R(y)})
\end{equation*}
is also a basic constructor
of $A/R$.

For the counit note that the canonical constructor 
$c_{\idtypevar{A}}:A\mapsto A/\idtypevar{A}$ of $A/\idtypevar{A}$ 
is an equivalence. Hence we define $\varepsilon_A:A/\idtypevar{A}\to A$ 
to be $c_{\idtypevar{A}}^{-1}$.

Note that $Q(\eta_{(A,R)}):A/R\to (A/R)/\idtypevar{A/R}$ is the unique map with the
the property that the square
\begin{equation*}
\begin{tikzcd}[column sep=large]
A \ar{r}{c_R} \dar[][swap]{c_R} & A/R \dar{c_{\idtypevar{A/R}}} \\ 
A/R \ar{r}[swap]{Q(\eta(A,R))} & (A/R)/\idtypevar{A/R}
\end{tikzcd}
\end{equation*}
and therefore we have a homotopy $Q(\eta_{(A,R)})\sim c_{\idtypevar{A/R}}$. By
definition we have $\varepsilon_{A/R}\defeq c_{\idtypevar{A/R}}^{-1}$, hence the
triangle identity $\id{\varepsilon_{A/R}\circ Q(\eta_{(A,R)})}
{\refl{A/R}}$ follows.

For the other triangle equality, note that the pair $i(\varepsilon_A)$ consists of
the function $i(\varepsilon_A)_0 \defeq  c_{\idtypevar{A}}^{-1}$ and the
function $i(\varepsilon_A)_1$ which is the canonical proof that equivalences
preserve path relations. We also have $\eta_{(A,\idtypevar{A})}$ given by the
function $\eta_{(A,\idtypevar{A}),\,0}\defeq c_{\idtypevar{A}}$ and the basic constructor
$\eta_{(A,\idtypevar{A}),\,1}\defeq  p_{\idtypevar{A}}$ witnessing that $\eta_{(A,R),\,0}$ preserves the path relation
Since $A/\idtypevar{A}$ is a set, it follows that $\eta_{(A,\idtypevar{A})}$ is just the
canonical proof that $\eta_{(A,\idtypevar{A})}$ preserves the path relation and
hence we get the other triangle equality.
\end{proof}

To prove that this is an equivalence of categories we would need to show that 
$\pairr{A/R,\idtypevar{A/R}}$ and $\pairr{A,R}$ are isomorphic setoids. 
The classical axiom of choice, $\choice{-1}$, is needed to obtain a section for the surjective
map $A\to A/R$.

\begin{defn}
A mere relation $R:A\to A\to\prop $ is said to be \emph{effective} if the square
\begin{equation*}
\begin{tikzcd}
\sm{x,y:A}R (x,y) \ar{r}{\pi_1} \ar{d}[swap]{\pi_2} & A \ar{d}{c_R} \\ 
A \ar{r}[swap]{c_R} & A/R
\end{tikzcd}
\end{equation*}
is a pullback square. 
\end{defn}

The following theorem uses univalence for mere propositions.

\begin{thm}\label{prop:sets_exact}
Suppose $\pairr{A,R}$ is an object of $\eqrel $. Then there is an
equivalence
\begin{equation*}
\eqv{(\id{c_R(x)}{c_R(y)})}{R(x,y)}
\end{equation*}
for any $x,y:A$. In other words, equivalence relations are effective.
\end{thm}

\begin{proof}
We begin by extending $R$ to a mere relation $\tilde{R}:A/R\to A/R\to\prop $. After
the construction of $\tilde{R}$ we will show that $\tilde{R}$ is equivalent
to the identity type on $A/R$. We define $\tilde{R}$ by double induction on
$A/R$ (note that $\prop $ is a set by univalence for mere propositions). We
define $\tilde{R}(c_R(x),c_R(y)) \defeq  R(x,y)$. For $r:R(x,x')$ and $s:R(y,y')$,
the transitivity and symmetry 
of $R$ gives an equivalence from $R(x,y)$ to $R(x',y')$. This completes the
definition of $\tilde{R}$. To finish the proof of the statement, we need
to show that $\eqv{\tilde{R}(w,w')}{(\id{w}{w'})}$ for every $w,w':A/R$. We can
do this by showing that the type $\sm{w':A/R}\tilde{R}(w,w')$ is contractible for
each $w:A/R$. We do this by induction. Let $x:A$. We have the element 
$\pairr{c_R(x),\rho(x)}:\sm{w':A/R}\tilde{R}(c_R(x),w')$, where $\rho$ is
the reflexivity term of $R$, hence we only
have to show that
\begin{equation*}
\prd{w':A/R}{r:\tilde{R}(c_R(x),w')}\id{\pairr{w',r}}{\pairr{c_R(x),\rho(x)}},
\end{equation*}
which we do by induction on $w'$. Let $y:A$ and let $r:R(x,y)$. Then we have
the path $\opp{p_R(r)}:\id{c_R(y)}{c_R(x)}$. We automatically get a path from
$\id{\trans{\opp{p_R(r)}}{r}}{\rho(x)}$, finishing the proof.
\end{proof}

\subsection{Voevodsky's impredicative quotients}\label{sec:resizing}
A second construction of quotients is due to Voevodsky~\cite{pelayo2013preliminary}.
He defines the quotient $A/R $ as the type of equivalence classes of $R$, 
i.e.\ as the image of $R$ in $A\to\prop $. This gives a direct
construction of quotients, but it requires a resizing rule. In this section we treat Voevodsky's
construction of quotients, but we note up front that our result 
that $\catset $ is a $\piw$-pretopos does not rely on the material
presented here. Throughout this section
we assume that $\pairr{A,R}$ is an object of $\eqrel $.

\begin{defn}
A predicate $P:A\to\prop $ is said to be an \emph{equivalence class} with 
respect to $R $ if there is an element of type
\begin{equation*}
\mathsf{isEqClass}(R ,P)\defeq \sm{x:A}\prd{y:A}\eqv{R(x,y)}{P(y)}
\end{equation*}
\end{defn}

\begin{defn}\label{def:VVquotient}
We define
\begin{equation*}
A\sslash R \defeq \sm{P:A\to\prop }\brck{\mathsf{isEqClass}(R,P)}.
\end{equation*}
\end{defn}

Using univalence for mere propositions, the following is a consequence of the definition:

\begin{lem}
The type $A\sslash R $ is equivalent to $\im(R)$, the image of $R:A\to(A\to \prop)$.
\end{lem}

In \autoref{prop:images_are_coequalizers} we have shown that images are
coequalizers. In particular, we immediately get the coequalizer diagram
\begin{equation*}
\begin{tikzcd}
\sm{x,y:A}\id{R(x)}{R(y)} \ar[yshift=.7ex]{r}{\pi_1}
\ar[yshift=-.7ex]{r}[swap]{\pi_2} & A \ar{r} & A\sslash R
\end{tikzcd}
\end{equation*}
We can use this to show that any equivalence relation is effective. First, we
demonstrate that kernels (of functions between sets) are effective equivalence
relations. 

\begin{thm}\label{prop:kernels_are_effective}
Let $f:A\to B$ between any two sets. Then
the relation $\ker(f):A\to A\to\type$ given by 
$\ker(f,x,y)\defeq \id{f(x)}{f(y)}$ is effective. 
\end{thm}

\begin{proof}
We will use \autoref{prop:images_are_coequalizers} which asserted that $\im(f)$ 
is the coequalizer of $\pi_1,\pi_2: (\sm{x,y:A}\id{f(x)}{f(y)})\to A$. This
gives an equivalence $\eqv{\im(f)}{A/\ker(f)}$. Note that the canonical kernel 
pair of the function 
$\tilde{f}$ defined in \autoref{defn:im-fact} consists 
of the two projections
\begin{equation*}
\pi_1,\pi_2:\big(\sm{x,y:A}\id{\tilde{f}(x)}{\tilde{f}(y)}\big)\to A.
\end{equation*}
For any $x,y:A$, we have equivalences
\begin{align*}
\id{\tilde{f}(x)}{\tilde{f}(y)} 
& \eqvsym 
\sm{p:\id{f(x)}{f(y)}}\id{\trans{p}{\sbproj{\pairr{x,\refl{f(x)}}}}}
{\sbproj{\pairr{y,\refl{f(y)}}}}\\ 
& \eqv{}{\id{f(x)}{f(y)}},
\end{align*}
where the last equivalence holds because 
$\sbrck{\hfib{f}b}$ is a mere proposition for any $b:B$. 
Therefore, we get that
\begin{equation*}
\eqv{\big(\sm{x,y:A}\id{\tilde{f}(x)}{\tilde{f}(y)}\big)}{\big(\sm{x,y:A}\id{f(x)}{f(y)}\big)}
\end{equation*}
and hence we may conclude that $\ker(f)$ is an effective relation.
\end{proof}

\begin{thm}
Equivalence relations are effective 
and $\eqv{A/R}{A\sslash R}$. 
\end{thm}

\begin{proof}
We need to analyze the coequalizer diagram
\begin{equation*}
\begin{tikzcd}
\sm{x,y:A}\id{R(x)}{R(y)} \ar[yshift=.7ex]{r}{\pi_1}
\ar[yshift=-.7ex]{r}[swap]{\pi_2} & A \rar & A\sslash R
\end{tikzcd}
\end{equation*}
By function extensionality, the type $\id{R(x)}{R(y)}$ is equivalent to the type of homotopies from $R (x)$ to $R (y)$, which by 
the univalence axiom is equivalent to $\prd{z:A}\eqv{R(x,z)}{R(y,z)}$. Since $R $ is an equivalence relation, the latter type 
is equivalent to $R (x,y)$. To
summarize, we get that $\eqv{(\id{R(x)}{R(y)})}{R(x,y)}$, so $R $ is effective since it is equivalent to an effective relation. Also,
the diagram
\begin{equation*}
\begin{tikzcd}
\sm{x,y:A}R(x,y) \ar[yshift=.7ex]{r}{\pi_1}
\ar[yshift=-.7ex]{r}[swap]{\pi_2} & A \rar & A\sslash R
\end{tikzcd}
\end{equation*}
is a coequalizer diagram. Since coequalizers are unique up to equivalence, it follows that $\eqv{A/R}{A\sslash R}$.
\end{proof}

One may wonder about the predicative interpretation of the quotient constructions above.
One could argue that the construction using higher inductive types 
is predicative by considering the interpretation of this
quotient in the setoid model~\cite{Altenkirch1999,coquand2012constructive}. 
In this model the quotient does not raise the universe level. 
A similar observation holds for constructions that can be carried out in the groupoid 
model~\cite{hofmann1998groupoid}. These observations should suffice for the set-level higher inductive 
types we use in the present paper.

We have an inclusion $\prop_{\type_i}\to\prop_{\type_{i+1}}$. The assumption that this map is an equivalence is called the 
\emph{propositional resizing axiom}; see~\cite{TheBook}. This form of impredicativity would make Voevodsky's quotient small.

The following replacement axiom is derivable from the propositional resizing axiom; see~\cite{Universe-poly}.
\begin{lem}
Let $\type$ be a universe and $X:\type$. Let $f$ be a surjection of $X$ onto a set $Y$. Then there exists a $Z:\type$ 
which is
equivalent to $Y$.
\end{lem}
\begin{proof}
Define $Z\defeq X\sslash\ker(f)$ using a map to the small mere propositions in \autoref{def:VVquotient}. 
Then $Z:\type$ and $\eqv{Z\defeq X/\ker(f)}{\im f}\eqv{}{Y}$.
\end{proof}

\subsection{The object classifier}\label{sec:object_classification}
One of the reasons that the definition of predicative topos in~\cite{vandenBerg2012} contains such a long
list of requirements is the absence of a small subobject classifier. Nevertheless, 
in Martin-L\"of type theory we have the possibility of considering
a tall hierarchy of universes nested in one another according to the ordering
of the hierarchy. While a universe with a subobject classifier would be impredicative,
there is no problem with subobject classifiers at higher universe levels, or simply
large subobject classifiers. 

In higher topos theory one considers not only subobject classifiers, 
which classify the monomorphisms, but also object classifiers which classify more general classes of maps.
In this section we will establish the existence of an internal analogue of such large object classifiers.

Moreover, we will find an $n$-object classifier for
every $n:\Nnegtwo$, where the $n$-object classifier will classify the
functions with $n$-truncated homotopy fibers. In \autoref{sec:prelim}
we saw that the monomorphisms are exactly the $(-1)$-truncated functions. Therefore,
the $(-1)$-object classifier will correspond to a (large) subobject
classifier.

In addition to the size issue of the object classifiers, we will see that the
$n$-object classifier will generally not be an \typele{n}, but an \typele{(n+1)}. 
This observation should be regarded in contrast to the theory of 
predicative toposes, where a universal small
map is required to exist. Such a universal small map is suggested to be a map
between sets, 
but it seems that within the current setting of homotopy type theory we cannot expect
such a map to exist. The main reason is that the universal small map of sets
will in general be a map of groupoids; a universal small map of groupoids will
in general be a map of $2$-groupoids, etc.

\begin{thm}\label{thm:nobject_classifier_appetizer}
For any type $B$ there is an equivalence
\begin{equation*}
\chi:\eqv{\big(\sm{A:\type}A\to B\big)}{B\to\type}.
\end{equation*}
Likewise, there is an equivalence
\begin{equation*}
\chi_n:\eqv{\big(\sm{A:\type}{f:A\to B}\prd{b:B}\istype{n}(\hfib{f}b)\big)}{(B\to \typele{n})}
\end{equation*}
for every $n:\Nnegtwo$.
\end{thm}

\begin{proof}
We begin by constructing the first equivalence, i.e.\ we have to construct functions
\begin{align*}
\chi & : (\sm{A:\type}A\to B)\to B\to\type\\
\psi & : (B\to\type)\to(\sm{A:\type}A\to B).
\end{align*}
The function $\chi$ is defined by $\chi(\pairr{A,f},b)\defeq \hfib{f}b$. The
function $\psi$ is defined by $\psi(P)\defeq \pairr{(\sm{b:B}P(b)),\proj1}$. Now
we have to verify that $\chi\circ\psi\sim\idfunc$ and that $\psi\circ\chi
\sim\idfunc$:
\begin{enumerate}
\item Let $P$ be a family of types over $B$. It is a basic fact 
(see Theorem 4.8.1 of~\cite{TheBook}) that 
$\eqv{\hfib{\proj1}{b}}{P(b)}$ and therefore it follows immediately
that $P\sim\chi(\psi(P))$.
\item Let $f:A\to B$ be a function. We have to find a path
\begin{equation*}
\id{\pairr{(\sm{b:B}\hfib{f}b),\proj1}}{\pairr{A,f}}
\end{equation*}
First note that we have the basic equivalence
$e:\eqv{(\sm{b:B}\hfib{f}b)}{A}$ with $e(b,a,p)\defeq a$ and $e^{-1}(a)
\defeq \pairr{f(a),a,\refl{f(a)}}$. It also follows that
$\id{\trans{e}{\proj1}}{\proj1\circ e^{-1}}$. From this, we immediately read off
that $\id{(\trans{e}{\proj1})(a)}{f(a)}$ for each $a:A$. 
\end{enumerate}
This completes the proof of the first of the asserted equivalences.

To find the second set of equivalences, note that if we restrict $\chi$ to
functions with $n$-truncated homotopy fibers we get a family of $n$-truncated
types. Likewise, if we restrict $\psi$ to a family of $n$-truncated types we get
a function with $n$-truncated homotopy fibers. To finish the proof we observe that truncatedness is a mere proposition, 
hence adding it as a restriction on both sides does not disturb the fact that the two functions are inverse equivalences.
\end{proof}

\begin{defn}
Define
\begin{equation*}
\pointed{\type}\defeq \sm{A:\type}A\qquad\text{and}\qquad \pointed{(\typele{n})}\defeq 
\sm{A:\typele{n}}A.
\end{equation*}
Thus, $\pointed{\type}$ stands for the \emph{pointed types} (by analogy with the pointed spaces) and $\pointed{(\typele{n})}$ stands for
the pointed $n$-types.
\end{defn}

The following theorem states that we have an object classifier.
\begin{thm}\label{thm:nobject_classifier}
Let $f:A\to B$ be a function. Then the diagram
\begin{equation*}
\begin{tikzcd}
A \ar{r}{\vartheta_f} \ar{d}[swap]{f} & \pointed{\type} \ar{d}{\proj1} \\ 
B \ar{r}[swap]{\chi_f} & \type
\end{tikzcd}
\end{equation*}
is a pullback diagram. Here, the function $\vartheta_f$ is defined by
\begin{equation*}
\lam{a}\pairr{\hfib{f}{f(a)},\pairr{a,\refl{f(a)}}}.
\end{equation*}
A similar statement holds when we replace $\type$ by $\typele{n}$.
\end{thm}
\begin{proof}
Note that we have the equivalences
\begin{align*}
A & \eqv{}{\sm{b:B}\hfib{f}b}\\
 & \eqv{}{\sm{b:B}{X:\type}{p:\id{\hfib{f}b}{X}}X}\\
 & \eqv{}{\sm{b:B}{X:\type}{x:X}\id{\hfib{f}b}{X}}\\
 & \jdeq B\times_{\type}\pointed{\type}.
\end{align*}
which gives us a composite equivalence $e:\eqv{A}{B\times_\type\pointed{\type}}$. 
We may display the action of this composite equivalence step by step by
\begin{align*}
a & \mapsto \pairr{f(a),\pairr{a,\refl{f(a)}}}\\
 & \mapsto \pairr{f(a),\hfib{f}{f(a)},\refl{\hfib{f}{f(a)}},\pairr{a,\refl{f(a)}}}\\
 & \mapsto \pairr{f(a),\hfib{f}{f(a)},\pairr{a,\refl{f(a)}},\refl{\hfib{f}{f(a)}}}
\end{align*}
Therefore, we get homotopies $f\sim\pi_1\circ e$ and $\vartheta_f\sim \pi_2\circ e$. 
\end{proof}

We include the following lemma because the domain of the subobject classifier
is usually the terminal object.

\begin{lem}\label{lem:subobject}
The type $\pointed{\prop}$ is contractible.
\end{lem}
\begin{proof}
Suppose that $\pairr{P,u}$ is an element of $\sm{P:\prop}P$. Then we have $u:P$ and hence there is an element of type $\iscontr (P)$. It
follows that $\eqv{P}{\unit}$ and therefore we get from the univalence axiom that there is a path
$\id{\pairr{P,u}}{\pairr{\unit,\ttt }}$.
\end{proof}

\subsection{$\catset$ is a \texorpdfstring{$\Pi$}{Π}$\mathsf{W}$-pretopos}
We assume the existence of the higher inductive types for truncation and quotients. The univalence axiom is used in
\autoref{epis-surj}, but not to prove that surjections are epimorphisms. We do use propositional univalence in
\autoref{prop:images_are_coequalizers}.
\begin{thm} The category $\catset$ is a \piw-pretopos.
\end{thm}
\begin{proof}
We have an initial object, disjoint finite sums (\autoref{lem:sumsdisjount}), and finite limits 
(\autoref{pb}). Sums are stable under pullback; see \autoref{prop:sums_stable}. So, $\catset$ is
lextensive. $\catset$ is locally cartesian closed. This follows from the preparations we made in \autoref{sec:prelim}, using the
fact that the existence of $\Pi$-types (and functional extensionality) gives local cartesian closure 
e.g.~\cite[Prop.~1.9.8]{jacobs1999categorical}.
The category $\catset$ is regular (\autoref{thm:set_regular}) and quotients are effective (\autoref{prop:sets_exact}). We thus have an
exact
category, since it is also lextensive, we have a pretopos. It has $\Pi$-types (\autoref{lem:hlevel_prod}) and W-types
(\autoref{lem:hlevel_w}), so we have a $\piw$-pretopos. 
\end{proof}

We conclude that the only thing which prevents $\catset$ from being a topos is
a small subobject classifier. $\prop$ is large, but if we
assume the resizing rules from \autoref{sec:resizing}, then $\prop$ becomes small. We have seen in
\autoref{thm:nobject_classifier} that it satisfies the properties of a subobject classifier and hence we actually
obtain a topos.

\section{Choice and collection axioms}\label{sec:multiplechoice}

In this section we study two axioms, the axiom of collection and the axiom of multiple choice, 
that extend beyond the basic set theory. It seems that these
axioms are not provable in homotopy type theory, but we investigate their relationship with the other axioms.
As a first attempt, we would try to investigate these axioms in the framework for algebraic set theory. However, in homotopy type theory 
the universe of sets is a
groupoid and hence none of the $1$-categorical frameworks quite satisfy our needs. Thus we follow algebraic set theory only loosely. We will
work in the naive `$\infty$-category' of types, but observe that most of the constructions either deal with sets or with maps that have
set-fibers. In these cases the constructions align with the constructions in the $1$-category $\catset$. In particular, we use $(-1)$-connected map
for cover and use (homotopy) pullbacks. We note that such a pullback between sets reduces to the $1$-pullback in the
$1$-category $\catset$.

In algebraic set theory one considers a category of classes and isolates the sets within them. In the present paper, 
next to this size issue, we are mostly concerned with the dimension, the complexity of the equality.
Although the results in this section may not be spectacular, we record to what extent the seemingly natural framework of 
algebraic set theory works and where it breaks down.

\subsection{Stable classes of maps}
\begin{defn}\label{defn:small_maps}
A class $\smal:\prd{X,Y:\type}(X\to Y)\to\prop$ is \emph{stable} (\cite[Def.\ 3.1]{MoerdijkPalmgren2002}) if it satisfies:
\begin{description}
\item[pullback stability] In a pullback diagram as below, we have $\smal(f)\to\smal(g)$.
\begin{equation*}
\begin{tikzcd}
X \arrow{r}{} \arrow{d}[swap]{g} & A \arrow{d}{f} \\
Y \arrow{r}[swap]{h} & B
\end{tikzcd}
\end{equation*}
\item[Descent] If $h$ in the above diagram is surjective, then $\smal(g)\to\smal(f)$. 
\item[Sum] If $\smal(f)$ and $\smal(g)$, then $\smal(f+g)$. Here $f+g$ is the obvious map between the disjoint sums.
\end{description}
A class of stable maps is \emph{locally full}~\cite[3.2]{MoerdijkPalmgren2002}
if for all $g:X\to Y$ and $f:Y\to Z$ for which $\smal(f)$ holds, 
we have $\smal(g)$ if and only if $\smal(f\circ g)$.

A class of maps $\smal$ is called a \emph{class of small maps} if it is stable, locally full and for each $X$, the category
$\smal_X$ --- the small maps over $X$ --- forms a $\piw$-pretopos, see~\cite[3.3]{MoerdijkPalmgren2002}.
\end{defn}

\begin{thm}\label{thm:small}
 The class of set-fibered maps is a class of small maps.
\end{thm}

\begin{proof}
The class of maps with set-fibers is stable (it even has dependent sums).

We prove the claim that the class of set-fibered maps is locally full. Consider
$g:X\to Y$ and a set-fibered map $f:Y\to Z$. If $g$ has set-fibers, then $f\circ g$ has set-fibers, as sets are
closed under $\Sigma$-types. For the converse, let $y:Y$. Note that $\hfib{g}{y}$
fits in the following diagram
\begin{equation*}
\begin{tikzcd}
\hfib{g}{y} \ar{r} \ar{d} & \hfib{f\circ g}{f(y)} \ar{r} \ar{d} & X \ar{d}{f} \\
\unit \ar{r} & \hfib{f}{f(y)} \ar{r} & Y
\end{tikzcd}
\end{equation*}
in which the outer rectangle and the square on the right are pullbacks. Hence
the square on the left is a pullback too. So we see that $\hfib{g}{y}$ is a
pullback of sets and consequently it is a set itself.

By the use of the object classifier, \autoref{thm:nobject_classifier}, we see that the category $\smal_X$ is equivalent
to the category of sets in context $X$ which forms a $\piw$-pretopos; see \autoref{sec:quotient}.
\end{proof}

\subsection{Representable classes of small maps}
\begin{defn}
A commuting diagram of the form 
\begin{equation*}
\begin{tikzcd}
X \arrow{r}{} \arrow{d}[swap]{g} & A \arrow{d}{f}\\
Y \arrow{r}[swap]{p} & B
\end{tikzcd}
\end{equation*}
is said to be a \emph{quasi-pullback} if the corresponding map from $X$ to $Y\times_B A$ is surjective. A quasi-pullback square in which
$p$ is surjective is said to be a \emph{covering square}. In this case, we say that $f$ is covered by $g$.
\end{defn}

\begin{defn}
A stable class $\smal$ of maps is said to be \emph{representable} if there exists a function $\pi:E\to U$
for which $\smal(\pi)$ holds and such that every function $f:A\to B$ satisfying $\smal(f)$ is covered by a pullback
of $\pi$. More explicitly, the latter condition means that we can fit $f$ in a diagram of the form
\begin{equation}\label{eq:small_repr}
\begin{tikzcd}
A \ar{d}[swap]{f} & X\times_U E \ar[->>]{l}[swap]{p_1} \ar{r}{} \ar{d}{\pi_1} & E \ar{d}{\pi} \\
B & X \ar[->>]{l}{p_0} \ar{r}[swap]{\chi} & U
\end{tikzcd}
\end{equation}
where the left hand square is covering and the square on the right is a pullback. 
\end{defn}

By \autoref{thm:nobject_classifier}, the class of set-fibered maps (in a universe $\type$) is representable (by a function from a larger 
universe). Moreover, we can take the left hand square to be the identity.

\subsection{The collection axiom}
\begin{defn}
A covering square
\begin{equation*}
\begin{tikzcd}
D \ar{r}{q} \ar{d}[swap]{g} & B \ar{d}{f} \\
C \ar{r}[swap]{p} & A
\end{tikzcd}
\end{equation*}
is said to be a \emph{collection square} if for any $a:A$, $E:\type$ and any surjective function $e:E\twoheadrightarrow\hfib{f}{a}$,
\begin{align*}
\exis{c:\hfib{p}{a}}{t:\hfib{g}{c}\to E}e\circ t\sim q_c;
\end{align*}
where $q_c$ is the restriction of $q$ to $\hfib{g}{a}$.
\end{defn}

\begin{defn}
Let $\smal$ be a class of small maps. The \emph{collection axiom} is the statement
$\mathsf{CA}(\smal)$ that for any small map $f:A\to X$ and any surjection $p:C\to A$ 
from a set $C$ there is a quasi-pullback diagram of the form
\begin{equation*}
\begin{tikzcd}
B \ar{r}{} \ar{d}[swap]{g} & C \ar[->>]{r}{p} & A \ar{d}{f} \\
Y \ar[->>]{rr}{} & & X
\end{tikzcd}
\end{equation*}
in which $g$ is a small map and the bottom map is surjective.
\end{defn}

Since we have an object classifier, we can replace maps with types in a context. After this transformation, the collection
axiom becomes:
\begin{equation*}
(\fall{a:A}\exis{c:C}R(a,c))\to\exis{B:\smal}{f:B\to C}\fall{a:A}\exis{b:B}R(a,f(b))
\end{equation*}
where $C$ is a set, $A:\smal$ and $R:A\to C\to\prop$.

This axiom is often included in the axioms for algebraic set theory and hence in predicative topos theory.
It seems unlikely that this axiom is provable in homotopy type theory, simply because none of its axioms seem applicable. 
In the constructive set theory CZF~\cite{aczel2001notes}, unlike in classical Zermelo set theory, the collection axiom is \emph{stronger}
than the replacement axiom. The replacement axiom \emph{is} derivable from the resizing rules; see \autoref{sec:resizing}.
In line with Voevodsky's proposal to add resizing rules to homotopy type theory, one could also consider its extension with the
collection axiom.

An cumulative hierarchy of sets may be defined using higher inductive types (see \cite{TheBook}). 
The induced set theory does satisfy the replacement axiom.

\subsection{The axiom of multiple choice}

\begin{defn}
The axiom of multiple choice (AMC) is the statement that any function fits on the right in a collection square.
\end{defn}

The axiom of multiple choice implies the collection axiom \cite[Prop.4.3]{MoerdijkPalmgren2002}. Conversely, by the existence of 
the object classifier, the collection axiom implies AMC. This follows from a small adaptation of \cite[Thm 4.3]{vandenBerg2012}.

It seems difficult to derive AMC, or equivalently the collection axiom, in a univalent type theory, even
when we add resizing rules. One possible route towards a counter model would be
to construct the Kan simplicial set model in ZF (without choice) and use the
fact that ZF does not prove AMC~\cite{vandenBerg2012}. However, this is beyond
the scope of this article.

\subsection{Projective covers}

The 0-truncated types in the cubical set model are precisely the setoids. Considering the cubical model inside an extensional type theory 
with a propositions-as-types interpretation, every set has a projective cover~\cite{van2009three}. The same holds if we have 
the axiom of choice in the meta-theory. Both the collection axiom \cite[5.2]{JoyalMoerdijk1995} and AMC follow from this 
axiom~\cite{MoerdijkPalmgren2002}. On the other hand~\cite[Prop.~11.2]{rezk2010toposes}, the $0$-truncation of a model topos of simplicial 
sheaves on a site is the topos of sheaves on that site. Hence, if we start with a topos without countable choice, we cannot have projective 
covers in the cubical sets model. This suggests investigating homotopy type theory with and without this axiom.
It is argued in~\cite{vandenBerg2012} that we need AMC to obtain a model theory for predicative toposes with good closure properties, e.g.\ 
closure under sheaf models. Concretely, AMC is used to show that W-types are small in sheaf models, but also to show that every internal 
site is presentable. It would be interesting to reconsider the latter of these issues in the presence of higher inductive types and the univalence axiom.

\section{Conclusion and outlook}

Our work is a contribution to the program of providing an elementary (first order) definition of an $\infty$-topos as conjectured models of
univalent homotopy type theory with higher inductive types~\cite{shulman2012univalence}. 
One would hope that many of the constructions that apply to predicative toposes (sheaves, realizability, gluing, ...) can be extended to 
homotopy type theory. By showing that $\catset$s form a predicative topos, we make a small step in this direction. Our result may be compared 
to e.g.~Proposition 11.2 in Rezk~\cite{rezk2010toposes}: the $0$-truncation of a model topos of simplicial sheaves on a site is the topos of 
sheaves on that site. Shulman~\cite{shulman2012univalence} shows that univalence is stable under gluing.

Moreover, this research program should contribute to a better understanding of the model theory of type theory based 
proof assistants such as Coq~\cite{Coq} and agda~\cite{agda} improving the set theoretical models; e.g.~\cite{werner1997sets}.
These proof assistants currently lack subset types, quotient types, functional extensionality, 
proof irrelevance for mere propositions, etc. As a first step one could provide a type theory which internalizes the setoid 
model~\cite{Altenkirch1999}. Unfortunately, this approach does not provide a proper treatment of the universe, as it is not a \typele{0}. 
Univalent homotopy type theory, considered as the internal type theory of simplicial, or cubical, sets, may be seen as an extension of this 
approach to include higher dimensional types such as universes.
As we have shown, homotopy type theory features: $\catset$ as predicative topos with $\prop$ as a (large) subobject classifier and the 
universe as an object classifier. So, it should facilitate the
formalization of mathematics, especially now that a computational interpretation of the univalence
axiom has been verified in a model~\cite{coquand2012constructive,Cubical}.
The semantics of small induction-recursion~\cite{malatesta2012small} depends on the set theoretic equivalence of the functor 
category and the slice category. The object classifier in \autoref{sec:object_classification} internalizes this. Moreover, 
it captures a similar kind of smallness; e.g.\ \autoref{thm:small}.


\paragraph*{Thanks}
Most of this paper was written when we were at the Institute for Advanced Study for the special year on Univalent Foundations.
We are greatly thankful to the participants of this year, especially to Steve Awodey, Thierry Coquand, Peter Lumsdaine, Mike Shulman and 
Vladimir Voevodsky.
The suggestions by the referees helped to greatly improve the presentation of the paper.



\bibliographystyle{plainurl}
\bibliography{refs}
\end{document}